\documentclass[reqno]{amsart}

\usepackage{amsmath,amssymb}
\usepackage[numbers,sort&compress]{natbib}
\usepackage{proof}
\usepackage{stmaryrd}
\usepackage{graphicx}
\usepackage{enumitem}
\usepackage{array}
\usepackage{mathrsfs}
\usepackage{hyperref}

\newcommand{\base}[1]{\mathscr{#1}}

\newcommand{\supp}{\Vdash}

\newcommand{\proves}{\vdash}

\newcommand{\setIndCons}{\mathrm{ICONS}}

\newcommand{\foralltwo}{\raisebox{1.5ex}{\rotatebox{180}{$\mathbb{A}$}}}

\newcommand{\existstwo}{\raisebox{1.5ex}{\rotatebox{180}{$\mathbb{E}$}}}

\newcommand{\atset}[1]{\mathbb{#1}}

\newcommand{\calculus}[1]{\mathsf{#1}}
\newcommand{\rn}[1]{\mathsf{#1}}

\newcommand{\efq}{\rn{EFQ}}
\newcommand{\irn}[1]{\rn{#1}\mathsf{I}}
\newcommand{\ern}[1]{\rn{#1}\mathsf{E}}

\newcommand{\flatop}[1]{#1^\flat}
\newcommand{\natop}[1]{#1^\sharp}

\newcommand{\BASIC}{\mathrm{BASIC}}
\newcommand{\HCONS}{\mathrm{HCONS}}
\newcommand{\con}[1]{\mathrm{con}(#1)}
\newcommand{\open}[1]{\mathrm{open}(#1)}
\newcommand{\concl}[1]{\mathrm{concl}(#1)}
\newcommand{\NI}{\calculus{NI}}
\newcommand{\NItwo}{\calculus{NI}^2}
\newcommand{\Ndash}{\proves_{\NI}}
\newcommand{\Ntwodash}{\proves_{\NItwo}}
\newcommand{\Ain}{\mathfrak{A}}
\newcommand{\fresh}[1]{\text{\small($#1$ fresh)}}

\numberwithin{equation}{section}
\newtheorem{definition}{Definition}[section]
\newtheorem{theorem}[definition]{Theorem}

\newtheorem{lemma}[definition]{Lemma}
\newtheorem{corollary}[definition]{Corollary}

\theoremstyle{definition}
\newtheorem{example}[definition]{Example}
\theoremstyle{remark}
\newtheorem{remark}[definition]{Remark}

\begin{document}

\title[Base-extension Semantics for Intuitionistic Predicate Logic]{Base-extension Semantics for\\ Intuitionistic Predicate Logic}

\author{Alexander V. Gheorghiu}
\address{School of Electronics and Computer Science, University of Southampton}
\email{a.v.gheorghiu@soton.ac.uk}

\author{David J. Pym}
\address{Institute of Philosophy and UCL,
University of London}
\email{david.pym@sas.ac.uk, d.pym@ucl.ac.uk}


\begin{abstract}
Sandqvist's base-extension semantics explains the logical constants of intuitionistic propositional logic by inference in bases: collections of rules over atomic sentences that record the inferential content of a state of affairs. 
We extend it to intuitionistic predicate logic, first- and second-order, reading the quantifiers substitutionally, and prove natural deduction sound and complete. Two ideas carry the extension. First, a rule of a base may require a premiss to be established for a fresh name, one about which neither the assumptions nor the base say anything. With such rules, Sandqvist's completeness proof extends to first-order logic, whether validity is taken over all bases or over finite ones. Second, the second-order quantifier ranges over the atomic predicates, and validity is support in comprehensive bases, where every expressible condition has a definitional proxy. This yields intuitionistic Henkin second-order logic with full comprehension, not full second-order semantics. 
\end{abstract}

\maketitle

\section{Introduction}\label{sec:introduction}

What does `Tammy is a vixen' mean? One answer points to the world: the sentence is true just in case the thing named `Tammy' is among the vixens. Another points to use, to what we may infer from the sentence and what we may infer it from. From `Tammy is a vixen' we may infer that Tammy is a fox and that she is female, and from these two we may infer it:
\[
\begin{gathered}
\infer{\text{Tammy is a fox}}{\text{Tammy is a vixen}}
\qquad
\infer{\text{Tammy is female}}{\text{Tammy is a vixen}}
\\[3ex]
\infer{\text{Tammy is a vixen}}{\text{Tammy is female} & \text{Tammy is a fox}}
\end{gathered}
\]
On the second answer, these rules give the meaning of `vixen', just as Gentzen's rules
\[
\infer{\phi}{\phi \land \psi} \qquad \infer{\psi}{\phi \land \psi} \qquad \infer{\phi \land \psi}{\phi & \psi}
\]
give the meaning of conjunction~\cite{Gentzen1969}.

The second answer is \emph{inferentialism}: the meaning of an expression is given by the role it plays in inference. Prawitz~\cite{Prawitz1971ideas} and Dummett~\cite{Dummett1991logical} developed it for logic, Brandom~\cite{brandom2009articulating} made it a general theory of meaning, and it may be read as one interpretation of Wittgenstein's principle of meaning as use~\cite{wittgenstein2009philosophical}. \emph{Proof-theoretic semantics} is the corresponding programme in logic. It takes the rules of inference as primary, and explains the meaning of the logical constants and the validity of arguments in terms of them~\cite{schroeder2006validity}.

\emph{Base-extension semantics} is a proof-theoretic semantics built on rules of the first kind, those for `vixen'~\cite{Sandqvist2015,Piecha2017definitional,Schroeder2016atomic}. A \emph{base} is a collection of such rules over atomic sentences, and it records the inferential content of a state of affairs. The semantics then says, by one clause for each logical constant, when a base \emph{supports} a formula. A base supports an atomic sentence when it derives it, and it supports $\phi \to \psi$ when every extension of it that supports $\phi$ supports $\psi$: the state of affairs would support $\psi$ were it to support $\phi$. The approach descends from Prawitz's notion of proof-theoretic validity~\cite{Prawitz1971ideas,schroeder2006validity}.

Sandqvist~\cite{Sandqvist2015} gave such clauses for intuitionistic propositional logic and proved natural deduction sound and complete for them. The choice of clauses matters, since the most direct ones do not give intuitionistic logic (Section~\ref{sec:support}). Sandqvist's clauses explain disjunction by what may be inferred from it, and Pym et al.~\cite{PymRitterRobinson2024} have shown that the resulting semantics is natural in the sense of categorical logic.

We work within this tradition, and the paper does not defend it. Whether base-extension semantics deserves the name `semantics', and how it relates to model-theoretic accounts of consequence, are questions about the programme as a whole~\cite{schroeder2006validity,Schroeder2007modelvsproof}. One point bears on how our results should be read. Support is defined from the rules of bases alone, by one clause for each logical constant, and no proof system for the logic enters the definition. A soundness and completeness theorem therefore relates two independently given relations, and, as the case of disjunction shows, it can fail.

This paper extends Sandqvist's semantics, and his completeness theorem, to the quantifiers of intuitionistic predicate logic. Each order of quantification raises its own obstacle, and the paper makes one contribution for each. For first-order quantification it introduces \emph{freshness-sensitive bases}, which make reasoning about arbitrary objects expressible at the level of atomic rules. For second-order quantification it introduces \emph{comprehensive bases} and \emph{definitional proxies}, which recover comprehension.

For the first-order quantifier, the natural reading in this setting is substitutional: a base supports $\forall x\phi$ when it supports every instance $\phi[x \mapsto a]$. This reading faces an objection, which Sandqvist raised against his own classical theory of the quantifiers. Having established that every named object has a property, we should not thereby conclude that everything does, since `for all we know, some nameless thing might not'~\cite[\S2.8.2]{Sandqvist2005inferentialist}.

The same difficulty appears as a technical obstacle to completeness. Sandqvist's proof simulates natural deduction inside a base, and the rule that introduces $\forall x\phi$ requires its premiss to be established for an \emph{arbitrary} object. An ordinary atomic rule cannot say this. Its premisses are sentences about the objects it names, so a rule from `Socrates is mortal' to `everything is mortal' is wrong, and adding one such rule for each name does not make it right. What the rule needs is a premiss about a name of which nothing in particular is assumed.

We meet both points with a richer notion of base used in the logic programming view of the semantics~\cite{GheorghiuSearch}. Its rules may contain variables, and they may carry \emph{constraints} requiring that the name put for a variable be \emph{fresh}. A fresh name must not occur in the assumptions on which the premiss rests, as in natural deduction, and it must not occur in the base itself (Section~\ref{sec:bases}). A base mentions only finitely many names, so fresh names always exist, and a name fresh for a base stands for an arbitrary object: whatever the base derives of it, the base derives of every name (Lemma~\ref{lem:renaming}). With these bases, Sandqvist's method extends to the quantifiers, and natural deduction for intuitionistic first-order logic is sound and complete for the semantics. Write $\Gamma \supp \zeta$ when every base that supports the formulae in $\Gamma$ supports $\zeta$, and $\Gamma \Ndash \zeta$ when $\zeta$ is derivable from $\Gamma$ in natural deduction. Then
\[
\Gamma \supp \zeta \quad\text{iff}\quad \Gamma \Ndash \zeta . \tag{Theorem~\ref{thm:A}}
\]
The equivalence remains true when bases are required to be finite, as in Sandqvist's own treatment of the quantifiers~\cite[\S3.1]{Sandqvist2005inferentialist}.

Second-order quantification adds a further problem. The target logic allows the second-order quantifier to be instantiated by an arbitrary formula, which is full \emph{comprehension}. Read substitutionally over predicate names, the quantifier ranges over too little, and comprehension fails (Example~\ref{ex:names-alone}).

Model theory faces the same choice in another form. Full semantics lets the quantifier range over all subsets of the domain. Giving content to `all subsets' requires a set-theoretic background~\cite{Vaananen2001}, and no proof system is complete for it~\cite[Ch.~5]{van1994logic}. Henkin semantics lets the quantifier range over a designated collection of subsets that is closed under definability, and it does admit a complete proof system~\cite{henkin1950completeness}.

We give the inferential analogue of Henkin's choice. The second-order quantifier ranges over the \emph{atomic predicates}; that is, over whatever a single atomic sentence with gaps in it says of some objects, and closure under definability becomes a condition on bases. A \emph{proxy} for a condition is an atomic predicate that the rules of the base make equivalent to that condition, in the base and in every extension of it. The three rules above make `vixen' a proxy for `female fox'. A base is \emph{comprehensive} when every condition expressible in the language has a proxy in it (Section~\ref{sec:second-order}). Validity for second-order logic is support in every comprehensive base.

Natural deduction for the intuitionistic version of Henkin's second-order logic, with full comprehension, is sound and complete for this notion of validity. Write $\Gamma \supp_{\mathrm{H}} \zeta$ when every comprehensive base that supports the formulae in $\Gamma$ supports $\zeta$, and $\Gamma \Ntwodash \zeta$ when $\zeta$ is derivable from $\Gamma$ in that system. Then
\[
\Gamma \supp_{\mathrm{H}} \zeta \quad\text{iff}\quad \Gamma \Ntwodash \zeta . \tag{Theorem~\ref{thm:B}}
\]

The two theorems concern different classes of bases. Theorem~\ref{thm:A} holds over all bases, and over all finite bases, whereas Theorem~\ref{thm:B} holds over comprehensive bases only. The restriction is needed, because over arbitrary bases the rule that instantiates the second-order quantifier by a formula is not sound (Example~\ref{ex:not-sound}). Table~\ref{tab:comparison} compares the two results point by point.

Theorem~\ref{thm:B} concerns Henkin's logic. It shows that the logic with full comprehension has a base-extension semantics, and it says nothing about consequence over all subsets of a domain. Section~\ref{sec:discussion} explains why no complete semantics could.

Two of the paper's philosophical claims are established mathematically: the substitutional clause for the first-order quantifier never amounts to an enumeration of names (Corollary~\ref{cor:fresh-instance}), and comprehension is validated by a clause that explains $\foralltwo X\phi$ in terms of strictly simpler judgements (Definition~\ref{def:support-2}). A third claim is an interpretation that we propose, namely that the impredicativity of comprehension lies in which bases are considered and not in what the quantifier means. We leave open which bases can acquire the definitions that comprehensiveness demands without altering their other content (Section~\ref{sec:discussion}).

Both completeness proofs follow Sandqvist's method, in which each formula receives an atomic sentence as its \emph{counterpart} and a base tailored to the inference at hand makes each counterpart behave as the formula does in natural deduction. For second-order logic, the counterparts are themselves the proxies that make the base comprehensive.

Section~\ref{sec:bases} introduces atomic rules with constraints, bases, and derivations, and proves that a name fresh for a base stands for an arbitrary object. Section~\ref{sec:support} gives the support clauses for first-order logic and shows that, for basic sentences, the rules of a base behave like assumptions. Section~\ref{sec:soundness} proves that natural deduction for first-order logic is sound, and Section~\ref{sec:completeness} that it is complete (Theorem~\ref{thm:A}).

Section~\ref{sec:second-order} turns to second-order logic, shows that quantifying over predicate names is too weak, and introduces atomic predicates, proxies, and comprehensive bases. Section~\ref{sec:soundness-2} proves soundness, where the work is to show that a proxy may replace its condition anywhere in a formula. Section~\ref{sec:completeness-2} proves completeness (Theorem~\ref{thm:B}) by constructing a comprehensive base that simulates the whole language. Section~\ref{sec:discussion} discusses Sandqvist's worry about substitutional quantification, the interpretation of comprehensiveness, the relation to full second-order semantics, and Dummett's demand that definitions be conservative.

The following conventions hold throughout. Further notation is introduced where it is first used.
\begin{itemize}[label=--]
  \item $a, b, c, e$ are individual constants, $x, y, z$ individual variables, and $P, Q, R$ predicate constants. The letters $e$ and $E$ are reserved for eigenconstants.
  \item $p, q, r, u$ are basic sentences, $\atset{T}$ is a finite set of basic sentences, and $Z$ is an atom variable.
  \item $\phi, \psi, \chi, \gamma, \zeta$ are formulae, and $\Gamma, \Delta$ are finite sets of closed formulae.
  \item $\base{B}, \base{C}, \base{D}$ are bases, $\rho$ is an atomic rule, and $\sigma$ is a substitution for the variables of a rule.
  \item $\tau$ is a renaming of constants, $\pi$ a bijective renaming, and $\theta$ an instantiation (Section~\ref{sec:soundness-2}).
  \item $\proves_{\base{B}}$ is derivability in the base $\base{B}$, $\supp_{\base{B}}$ is support in $\base{B}$, and $\Ndash$ and $\Ntwodash$ are derivability in natural deduction.
\end{itemize}

\section{Bases}\label{sec:bases}

Base-extension semantics explains the logical constants in terms of something more basic: what can be inferred in a state of affairs, described by sentences that contain no logical constants at all. A \emph{base} records this inferential content as a collection of rules over atomic sentences. Since a state of affairs enters the semantics only through its inferential content, we henceforth identify a state of affairs with its base, and speak only of bases. This section defines the bases used throughout the paper.

Our bases differ from Sandqvist's~\cite{Sandqvist2015} in two respects, both forced by the quantifiers. First, a rule may be \emph{general}: it may contain variables and stand for all of its instances. Second, a rule may require that one of its premisses be established \emph{for an arbitrary object}. The second feature is the base-level analogue of the eigenvariable condition of natural deduction, and the completeness proof of Section~\ref{sec:completeness} depends on it. The resulting notion of base is the one introduced in \emph{Support is Search}~\cite{GheorghiuSearch} for logic programs, transposed here to natural deduction.

We fix denumerably infinite sets of \emph{individual constants} $a, b, c, \ldots$, which we collect in $\setIndCons$; of \emph{individual variables} $x, y, z, \ldots$; and, for each $n \geq 0$, of \emph{predicate constants} $P, Q, R, \ldots$ of arity $n$. Terms are individual constants and individual variables. Formulae are given by the grammar
\[
\phi ::= P(t_1, \ldots, t_n) \mid \bot \mid \phi \land \phi \mid \phi \lor \phi \mid \phi \to \phi \mid \forall x \phi \mid \exists x \phi
\]
where $P$ has arity $n$ and $t_1, \ldots, t_n$ are terms. All the connectives are primitive, and $\neg\phi$ abbreviates $\phi \to \bot$.

A formula is \emph{closed} if it has no free variables; it may contain constants. Formulae are identified up to renaming of bound variables. We write $\phi[x \mapsto t]$ for the result of substituting the term $t$ for the free occurrences of $x$ in $\phi$, renaming bound variables where necessary~\cite{van1994logic}. For a formula, a rule, or a set of these, $\con{-}$ denotes the set of individual constants occurring in it.

The second-order language extends this one, and Section~\ref{sec:second-order} says what changes in the material of the present section.

Rules are built from basic sentences. A \emph{basic sentence} is a closed atomic formula $P(a_1, \ldots, a_n)$, and $\BASIC$ is the set of basic sentences. For example, a base might contain the basic sentence `Socrates is human' as an axiom, together with the rule
\[
\infer{\text{Socrates is mortal}}{\text{Socrates is human}}
\]
In such a base, `Socrates is mortal' is derivable.

Following Sandqvist~\cite{Sandqvist2015}, rules may also discharge hypotheses, as in natural deduction. A rule such as
\[
\infer{r}{p & \deduce{q}{[s]}}
\]
says that $r$ may be concluded from $p$ together with a derivation of $q$ from the hypothesis $s$, the hypothesis being discharged by the rule. Rules of this kind are essential to Sandqvist's completeness theorem for intuitionistic logic~\cite{Sandqvist2015}, and they are equally essential here.

The rule from `Socrates is human' to `Socrates is mortal' is not, however, about Socrates in particular. \emph{Whoever} is human is mortal, and a base should be able to record this as a single rule, not as an infinite list of rules, one for each name. We therefore allow rules to contain variables, which stand for arbitrary names:
\[
\infer{M(x)}{H(x)}
\]
Here $H$ and $M$ are the predicates `is human' and `is mortal'. A rule with variables stands for all of its \emph{instances}, the rules obtained by putting constants for the variables. These are what Sandqvist calls \emph{open rules}~\cite[\S3.1]{Sandqvist2005inferentialist}.

General rules are not yet enough for the quantifiers. Suppose the language also contains the basic sentence $U$, `mortality is universal'. The mortality of Socrates does not license it, and neither does the mortality of any list of named individuals. What licenses $U$ is a derivation showing that an arbitrary object is mortal.

No rule of the kinds considered so far records this. The general rule from $M(x)$ to $U$ has the instance from $M(s)$ to $U$, where $s$ names Socrates, and so licenses the inference from `Socrates is mortal' to `mortality is universal'. The premiss of such a rule is always a sentence about a named object, and what is known of a named object may be special to it. The rule we want must demand that its premiss be established for an \emph{arbitrary} object. We write it with a side condition:
\begin{equation}\label{eq:U-rule}
\infer{U}{M(y)} \qquad \fresh{y}
\end{equation}

The side condition is a \emph{constraint}. It says that the name put for $y$ must be one about which nothing in particular is assumed, and it plays the role that the eigenvariable condition plays for the introduction rule of the universal quantifier in natural deduction. Indeed, Section~\ref{sec:completeness} shows that in every base containing the rule \eqref{eq:U-rule} and the converse rule from $U$ to $M(x)$, the basic sentence $U$ is derivable exactly when the formula $\forall x M(x)$ is supported.

\begin{definition}[Atomic rule]\label{def:atomic-rule}
An \emph{atomic rule} is a natural-deduction figure, possibly with a constraint:
\[
\infer[\rho]{r}{\deduce{p_1}{[\Sigma_1]} & \cdots & \deduce{p_n}{[\Sigma_n]}}
\qquad \fresh{\vec y}
\]
The \emph{premisses} $p_1, \ldots, p_n$ and the \emph{conclusion} $r$ are atoms, and each $\Sigma_i$ is a finite set of atoms, the \emph{hypotheses} discharged at the $i$-th premiss (omitted when empty). These atoms may contain individual variables, standing for individual constants, and \emph{atom variables} $Z, Z', \ldots$, standing for basic sentences. The \emph{constraint} is a set $\vec y$ of individual variables, omitted when empty. Each constrained variable occurs in at most one premiss position, that is, in at most one of the pairs $(\Sigma_i, p_i)$, and none occurs in $r$. (A constrained variable occurring nowhere imposes no condition; such rules arise from vacuous quantifiers.) An \emph{instance} of $\rho$ is obtained by substituting an individual constant for each individual variable and a basic sentence for each atom variable.
\end{definition}

The variables of a rule are implicitly quantified over the whole rule, so no binders are needed. Unconstrained variables are \emph{schematic}: when the rule is applied, its premisses must be established for the particular names that the instance supplies. Constrained variables are the eigenvariables of the base. The constraint does not restrict \emph{which} instances there are. It restricts how an instance may be \emph{used}: the constants put for the constrained variables must be fresh, in the sense of condition~(G) below.

\begin{remark}[Variables]\label{rem:convenience}
Writing rules with variables is a convenience, not a necessity. A rule without a constraint is a finite description of the set of its instances, each of which is a rule in Sandqvist's sense~\cite{Sandqvist2015}. A constrained rule is the set of its instances with a side condition on how they may be applied. What the variables buy is a simple notion of freshness, given next: a base mentions only finitely many names, and every other name is fresh for it.
\end{remark}

\begin{definition}[Base]\label{def:base}
A \emph{base} is a set $\base{B}$ of atomic rules mentioning only finitely many individual constants, that is, with $\con{\base{B}}$ finite. A constant is \emph{fresh} for $\base{B}$ if it does not belong to $\con{\base{B}}$. A base $\base{C}$ \emph{extends} $\base{B}$ if $\base{B} \subseteq \base{C}$.
\end{definition}

Every finite set of atomic rules is a base. For example, the axiom and the general rule above form the base
\[
\base{A} := \left\{\; \infer{H(s)}{} \;,\quad \infer{M(x)}{H(x)} \;\right\}
\]
in which $s$ stands for `Socrates'.

The finiteness requirement gives `fresh' its meaning. Rules with variables may give a base general content, about all objects. What a base cannot have is particular content about infinitely many names, as an infinite list of axioms $L(a_1, a_2), L(a_2, a_3), \ldots$ would give it. A name fresh for a base is therefore one about which the base says nothing in particular, and anything the base lets us derive about it, the base lets us derive about any object whatever (Lemma~\ref{lem:renaming}).

Whether a name may be treated as arbitrary depends on what the premiss \emph{rests on}. Suppose $M(s)$ has been obtained by the general rule of $\base{A}$ from the assumption $H(s)$. This derivation establishes the mortality of Socrates on an assumption about Socrates, and it tells us nothing about an arbitrary object. Had the assumption $H(s)$ been discharged, or had there been none, matters would stand differently.

The derivability relation alone does not record this information at the point where it is needed. A constrained rule is applied inside a larger derivation, and it must inspect the assumptions that are open at that point, some of which may be discharged further down. We therefore work with derivations as explicit trees, as in natural deduction, each carrying its set of open assumptions.

\begin{definition}[Derivation]\label{def:derivation}
Derivations $\mathcal{D}$, each with a \emph{conclusion} $\concl{\mathcal{D}} \in \BASIC$ and a finite set $\open{\mathcal{D}} \subseteq \BASIC$ of \emph{open assumptions}, are generated as follows.
\begin{itemize}[label=--]
  \item For each basic sentence $p$, the one-node tree $p$ is a derivation, with conclusion $p$ and open assumptions $\{p\}$.
  \item Let $\rho$ be an atomic rule as in Definition~\ref{def:atomic-rule}, and write $\vec y_i$ for the constrained variables occurring in its $i$-th premiss $p_i$ or in its hypotheses $\Sigma_i$. Let $\sigma$ substitute constants and basic sentences for the unconstrained variables of $\rho$. For each $i$, let $\vec e_i$ be a list of distinct constants matching $\vec y_i$, and let $\mathcal{D}_i$ be a derivation whose conclusion is the $i$-th premiss of the resulting instance, $p_i\sigma[\vec y_i \mapsto \vec e_i]$. Then
  \[
  \infer[\rho\sigma]{r\sigma}{\deduce{p_1\sigma[\vec y_1 \mapsto \vec e_1]}{\mathcal{D}_1} & \cdots & \deduce{p_n\sigma[\vec y_n \mapsto \vec e_n]}{\mathcal{D}_n}}
  \]
  is a derivation. Its conclusion is $r\sigma$, and its open assumptions are those of the $\mathcal{D}_i$ other than the hypotheses discharged at the $i$-th premiss:
  \[
  \bigcup_i \big(\open{\mathcal{D}_i} \setminus \Sigma_i\sigma[\vec y_i \mapsto \vec e_i]\big).
  \]
  The constants $\vec e_i$ are the \emph{eigenconstants} of this application of $\rho$, and $\rho\sigma$ denotes the partial instance, in which the constrained variables are still variables.
\end{itemize}
A derivation is \emph{legal in} a base $\base{B}$ if every rule applied in it is an instance of a rule of $\base{B}$ and, at each application, every eigenconstant $e$ of the $i$-th premiss satisfies the following \emph{freshness condition}~(G):
\begin{enumerate}[label=(G\arabic*)]
  \item\label{G1} $e$ does not occur in any open assumption of $\mathcal{D}_i$ other than the hypotheses discharged at the $i$-th premiss;
  \item\label{G2} $e$ does not occur in the partial instance $\rho\sigma$;
  \item\label{G3} $e$ does not occur in $\base{B}$.
\end{enumerate}
For a finite set $\atset{T}$ of basic sentences and a basic sentence $p$, we write $\atset{T} \proves_{\base{B}} p$ if some derivation legal in $\base{B}$ has conclusion $p$ and all its open assumptions in $\atset{T}$. We write $\proves_{\base{B}} p$ when $\atset{T}$ is empty.
\end{definition}

For rules without constraints this is derivability in Sandqvist's sense: a derivation applies instances of rules of the base, and $\atset{T} \proves_{\base{B}} p$ holds just when $p$ can be obtained from the assumptions $\atset{T}$. For example, in the base $\base{A}$,
\[
\infer{M(s)}{\infer{H(s)}{}}
\]
is a legal derivation with no open assumptions, so $\proves_{\base{A}} M(s)$.

Condition~(G) spells out what it is for a premiss to be established for an arbitrary object. Each part excludes one way in which the eigenconstant could fail to be arbitrary.
\begin{itemize}[label=--]
  \item Part~\ref{G1} is the familiar eigenvariable condition of natural deduction: the premiss must not rest on an assumption about $e$.
  \item Part~\ref{G2} corresponds to the requirement, in the rule that introduces $\forall x\phi$, that the eigenvariable not occur in $\forall x\phi$: the rule instance itself must not be about $e$.
  \item Part~\ref{G3} is new: the base must not be about $e$. This condition is needed because the rules of a base behave like assumptions (Lemma~\ref{lem:conservative}), so if $e$ is to be arbitrary, the base must contain no specific inferences about $e$.
\end{itemize}
Condition~(G) is the clause \textsc{Generic} of \emph{Support is Search}~\cite{GheorghiuSearch}, which requires the eigenconstant to be fresh for the whole program at that point. The next two examples show what the condition excludes.

\begin{example}[Open assumptions]\label{ex:G1}
Consider the tree
\[
\infer{U}{M(s)}
\]
which applies the rule~\eqref{eq:U-rule} to the one-node derivation $M(s)$, with eigenconstant $s$. Its conclusion is $U$ and its only open assumption is $M(s)$. It is not legal in any base, since the eigenconstant $s$ occurs in an open assumption of the premiss, violating~\ref{G1}. If it were legal, we should have $M(s) \proves_{\base{B}} U$ in every base $\base{B}$ containing~\eqref{eq:U-rule}: mortality would be universal on the assumption that Socrates is mortal.
\end{example}

\begin{example}[Freshness for the base]\label{ex:fresh}
Part~\ref{G3} cannot be dropped. Consider two bases, each containing the rule~\eqref{eq:U-rule} and the general rule of $\base{A}$:
\[
\begin{gathered}
\base{B}_1 := \left\{\; \infer{H(x)}{} \;,\quad \infer{M(x)}{H(x)} \;,\quad \infer{U}{M(y)}\ \fresh{y} \;\right\}
\\[2mm]
\base{B}_2 := \left\{\; \infer{H(s)}{} \;,\quad \infer{M(x)}{H(x)} \;,\quad \infer{U}{M(y)}\ \fresh{y} \;\right\}
\end{gathered}
\]
In $\base{B}_1$ everything is human, and in $\base{B}_2$ only Socrates is known to be human. Now consider the derivation
\[
\infer{U}{\infer{M(s)}{\infer{H(s)}{}}}
\]
in which $s$ is the eigenconstant of the last step.

The derivation is legal in $\base{B}_1$: the premiss rests on no assumptions, no constant occurs in $U$, and none occurs in $\base{B}_1$. So $\proves_{\base{B}_1} U$. In $\base{B}_2$ the same tree is not legal, because $s$ occurs in $\base{B}_2$, violating~\ref{G3}. Since no other name is human according to $\base{B}_2$, the basic sentence $U$ is not derivable there at all.

The two applications cannot be told apart by looking at the derivation, which is the same tree in both cases. The difference lies in the base. In $\base{B}_1$ the premiss $M(s)$ was obtained by reasoning that works for any object. In $\base{B}_2$ it was obtained from an axiom about Socrates.
\end{example}

Two properties of derivability underlie everything that follows. The first makes precise the claim that a name fresh for a base stands for an arbitrary object. The second says that whatever a base derives, every extension of it derives.

A \emph{renaming} is a map $\tau : \setIndCons \to \setIndCons$, not necessarily injective, which acts on formulae, rules, and sets of these by replacing every occurrence of each constant $a$ by $\tau(a)$. The Renaming Lemma says that a renaming carries derivations to derivations. It yields the monotonicity of derivability (Lemma~\ref{lem:mono}), it yields the invariance of support under permutations of names (Lemma~\ref{lem:equiv-supp}), and it is the precise form of the reply to Sandqvist's worry about nameless objects (Section~\ref{sec:discussion}).

\begin{lemma}[Renaming]\label{lem:renaming}
Let $\tau$ be a renaming. If $\mathcal{D}$ is a derivation legal in $\base{B}$, with conclusion $p$ and open assumptions $\atset{T}$, then there is a derivation of the same height, legal in $\tau\base{B}$, with conclusion $\tau p$ and open assumptions contained in $\tau\atset{T}$.
\end{lemma}

The open assumptions are \emph{contained} in $\tau\atset{T}$, not equal to it, because a noninjective $\tau$ may identify an open assumption with a hypothesis that is discharged.

The proof applies $\tau$ to every node of $\mathcal{D}$. The one difficulty is that $\tau$ may send an eigenconstant to a name that is not fresh. We avoid it by first exchanging each eigenconstant for a new one, which condition~(G) permits.

\begin{proof}
By induction on the height of $\mathcal{D}$, for all renamings $\tau$ at once. A one-node derivation $p$ is sent to the one-node derivation $\tau p$.

Suppose $\mathcal{D}$ ends with an application of $\rho\sigma$ to $\mathcal{D}_1, \ldots, \mathcal{D}_n$ with eigenconstants $\vec e_i$. Write $\atset{T}_i$ for the open assumptions of $\mathcal{D}_i$ other than the hypotheses discharged at the $i$-th premiss. The renamed instance is an instance of the rule $\tau\rho$ of $\tau\base{B}$.

For each $i$, choose distinct constants $\vec e_i'$ that occur neither in $\tau\base{B}$, nor in $\tau(\rho\sigma)$, nor in $\tau\atset{T}_i$. This is possible because all three are finite. Let $\tau_i$ agree with $\tau$ except that it sends $\vec e_i$ to $\vec e_i'$. By~(G), the constants $\vec e_i$ occur neither in $\base{B}$, nor in $\rho\sigma$, nor in $\atset{T}_i$, so $\tau_i$ and $\tau$ agree on all three.

The induction hypothesis, applied to $\mathcal{D}_i$ and $\tau_i$, gives a derivation legal in $\tau_i\base{B} = \tau\base{B}$ whose conclusion is $\tau_i$ applied to the $i$-th premiss, that is, the $i$-th premiss of the renamed instance with $\vec e_i'$ for the constrained variables. Its open assumptions are contained in $\tau_i\open{\mathcal{D}_i}$. Once the renamed hypotheses are discharged, the remainder is contained in $\tau_i\atset{T}_i = \tau\atset{T}_i$.

Applying the renamed rule to these derivations, with eigenconstants $\vec e_i'$, gives the required derivation. Condition~(G) holds by the choice of the $\vec e_i'$, and the open assumptions are contained in $\bigcup_i \tau\atset{T}_i = \tau\atset{T}$.
\end{proof}

Two special cases are used repeatedly. First, if $e$ is fresh for $\base{B}$, then the renaming $[e \mapsto a]$ leaves $\base{B}$ unchanged, so
\[
\atset{T} \proves_{\base{B}} p \quad\text{implies}\quad \atset{T}[e \mapsto a] \proves_{\base{B}} p[e \mapsto a] \qquad\text{for every constant } a.
\]
In words: whatever a base lets us derive about a name it does not mention, it lets us derive about every name, mentioned or not. Second, an eigenconstant can always be exchanged for any other constant that is fresh for the relevant data.

The second property, monotonicity, is needed because support will be defined by quantifying over the extensions of a base, so that a derivation legal in a base must remain available in every extension. This is less obvious than it looks, since a new rule may mention a constant that some derivation used as an eigenconstant. For instance, the derivation of $U$ in Example~\ref{ex:fresh} is not legal in $\base{B}_1 \cup \{\infer{H(s)}{}\}$, since $s$ is no longer fresh. The same derivation with another name in place of $s$ is legal there, and the proof of monotonicity makes this exchange in general.

\begin{lemma}[Monotonicity of derivability]\label{lem:mono}
If $\atset{T} \proves_{\base{B}} p$ and $\base{C}$ is a base extending $\base{B}$, then $\atset{T} \proves_{\base{C}} p$, by a derivation of the same height.
\end{lemma}
\begin{proof}
By induction on the height of a derivation $\mathcal{D}$ legal in $\base{B}$. A one-node derivation is legal in every base.

Suppose $\mathcal{D}$ ends with an application of $\rho\sigma$ with eigenconstants $\vec e_i$, and let $\atset{T}_i$ be as in the proof of Lemma~\ref{lem:renaming}. The eigenconstants $\vec e_i$ may occur in $\base{C}$. Choose distinct constants $\vec e_i'$ fresh for $\base{C}$ and occurring neither in $\rho\sigma$ nor in $\atset{T}_i$.

The renaming $[\vec e_i \mapsto \vec e_i']$ leaves $\base{B}$, $\rho\sigma$, and $\atset{T}_i$ unchanged, by~(G). So Lemma~\ref{lem:renaming} turns $\mathcal{D}_i$ into a derivation legal in $\base{B}$, of the same height, that establishes the $i$-th premiss with $\vec e_i'$ in place of $\vec e_i$ and rests on the correspondingly renamed assumptions. The induction hypothesis makes this derivation legal in $\base{C}$.

Applying $\rho\sigma$ with eigenconstants $\vec e_i'$ is legal in $\base{C}$ by the choice of the $\vec e_i'$, and the open assumptions are unchanged.
\end{proof}

\section{Support}\label{sec:support}

Section~\ref{sec:bases} gives the inferential content of a base: what can be derived in it. This section says what a base supports once logical constants are in play. Support is a judgement $\supp_{\base{B}} \phi$, read `the base $\base{B}$ supports $\phi$', and it gives the meaning of the logical constants of the first-order language. It is defined by clauses that say, for each logical constant, when a base supports a formula built with it. Except for the quantifiers, the clauses are Sandqvist's~\cite{Sandqvist2015}. What is new is the notion of base on which they operate.

A base supports a basic sentence just when it derives it:
\[
    \supp_{\base{B}} p \quad\text{iff}\quad \proves_{\base{B}} p 
\]
Conjunction is equally direct: $\phi \land \psi$ is supported when both $\phi$ and $\psi$ are.

Implication is read subjunctively: a base supports $\phi \to \psi$ if, were it enriched so as to support $\phi$, it would also support $\psi$. This is expressed by quantifying over \emph{extensions} of the base. For a nonempty finite set $\Delta$ of formulae, we put
\[
\begin{aligned}
\Delta \supp_{\base{B}} \phi \quad\text{iff}\quad &\text{for every base } \base{C} \text{ extending } \base{B},\\ &\text{if } \supp_{\base{C}} \delta \text{ for every } \delta \in \Delta, \text{ then } \supp_{\base{C}} \phi ,
\end{aligned}
\]
and say that $\phi \to \psi$ is supported in $\base{B}$ when $\phi \supp_{\base{B}} \psi$. 

Following Sandqvist, disjunction and absurdity are explained by what may be inferred from them. A base supports $\phi \lor \psi$ when every basic sentence that would follow from $\phi$, and would also follow from $\psi$, follows outright, in the present base and in every extension of it. A base supports absurdity when it derives every basic sentence.

The more obvious clause for disjunction, that $\phi \lor \psi$ is supported when $\phi$ is or $\psi$ is, does not give intuitionistic logic. As Piecha, Schroeder-Heister, and others showed, it leads to a logic validating the Kreisel--Putnam principle~\cite{SanzPiechaSchroeder2015,Piecha2016,SchroederHeister2019,Stafford2021}. The clause in elimination form is what yields intuitionistic logic~\cite{Sandqvist2015}.

The universal quantifier is read substitutionally. Consider the statement `every human is mortal'. To accept it is ordinarily to accept, for any name we may use, the corresponding instance: for `Socrates', the instance `if Socrates is human, then Socrates is mortal'. Accordingly, $\forall x \phi$ is supported when $\phi[x \mapsto a]$ is supported for every individual constant $a$.

Here model-theoretic and proof-theoretic semantics diverge: in model-theoretic accounts the quantifiers range over the elements of a domain, and in proof-theoretic semantics over the names of the language. Section~\ref{sec:discussion} returns to the worry that this leaves out nameless objects.

The existential quantifier is to the universal what disjunction is to conjunction: $\exists x\phi$ is supported when every basic sentence that would follow from each instance $\phi[x \mapsto a]$ follows outright.

These clauses define support by recursion on a measure of complexity. The \emph{degree} of a formula is given as follows: atoms have degree $1$ and $\bot$ has degree $2$; $\deg(\phi \circ \psi) := \deg\phi + \deg\psi + 1$ for $\circ \in \{\land, \lor, \to\}$; $\deg(\forall x \phi) := \deg\phi + 1$; and $\deg(\exists x\phi) := \deg\phi + 2$. Substituting a constant for a variable does not change the degree.

\begin{definition}[Support]\label{def:support}
Let $\base{B}$ be a base, let $\phi$ and $\psi$ be closed formulae, and let $\Delta$ be a nonempty finite set of closed formulae. \emph{Support} is defined by the clauses of Figure~\ref{fig:support}, in which $p$ ranges over basic sentences and $a$ over individual constants. We write $\emptyset \supp_{\base{B}} \phi$ for $\supp_{\base{B}} \phi$. A closed formula $\phi$ is a \emph{consequence} of a finite set $\Gamma$ of closed formulae, written $\Gamma \supp \phi$, if $\Gamma \supp_{\base{B}} \phi$ for every base $\base{B}$; $\phi$ is \emph{valid} if $\emptyset \supp \phi$.
\end{definition}

\begin{figure}[t]
\hrule\medskip
\[
\begin{array}{l@{\quad}c@{\quad}>{\raggedright\arraybackslash$}p{0.55\linewidth}<{$}@{\quad}r}
\supp_{\base{B}} p & \text{iff} & \proves_{\base{B}} p & (\mathrm{At})\\[1mm]
\supp_{\base{B}} \bot & \text{iff} & \proves_{\base{B}} p \text{ for every } p & (\bot)\\[1mm]
\supp_{\base{B}} \phi \land \psi & \text{iff} & \supp_{\base{B}} \phi \text{ and } \supp_{\base{B}} \psi & (\land)\\[1mm]
\supp_{\base{B}} \phi \lor \psi & \text{iff} & \text{for every } \base{C} \supseteq \base{B} \text{ and every } p, \text{ if } \phi \supp_{\base{C}} p \text{ and } \psi \supp_{\base{C}} p, \text{ then } \supp_{\base{C}} p & (\lor)\\[1mm]
\supp_{\base{B}} \phi \to \psi & \text{iff} & \phi \supp_{\base{B}} \psi & (\to)\\[1mm]
\supp_{\base{B}} \forall x \phi & \text{iff} & \supp_{\base{B}} \phi[x \mapsto a] \text{ for every } a & (\forall)\\[1mm]
\supp_{\base{B}} \exists x \phi & \text{iff} & \text{for every } \base{C} \supseteq \base{B} \text{ and every } p, \text{ if } \phi[x \mapsto a] \supp_{\base{C}} p \text{ for every } a, \text{ then } \supp_{\base{C}} p & (\exists)\\[1mm]
\Delta \supp_{\base{B}} \phi & \text{iff} & \text{for every } \base{C} \supseteq \base{B}, \text{ if } \supp_{\base{C}} \delta \text{ for every } \delta \in \Delta, \text{ then } \supp_{\base{C}} \phi & (\mathrm{Inf})
\end{array}
\]
\medskip\hrule
\caption{Support}\label{fig:support}
\end{figure}

The definition is a well-founded recursion. Measure a judgement $\Delta \supp_{\base{B}} \phi$ by the sum of the degrees of the members of $\Delta$ and of $\phi$, a basic sentence $p$ counting with degree $1$. Each judgement occurring on the right-hand side of a clause has smaller measure than the judgement on the left.

\begin{example}[The vixen base]\label{ex:vixen-support}
Return to the vixen of the introduction, and write $V$, $F_e$, and $F_o$ for `is a vixen', `is female', and `is a fox'. Consider the base
\[
\base{V} := \left\{\; \infer{V(t)}{} \;,\quad \infer{F_e(x)}{V(x)} \;,\quad \infer{F_o(x)}{V(x)} \;,\quad \infer{V(x)}{F_e(x) & F_o(x)} \;\right\}
\]
in which $t$ names Tammy. By (At) and $(\land)$, $\base{V}$ supports $F_e(t) \land F_o(t)$. It also supports $\forall x(V(x) \to F_o(x))$. By $(\forall)$ and $(\to)$ we must show, for each constant $a$, that every extension of $\base{V}$ deriving $V(a)$ derives $F_o(a)$, and this holds by the second general rule.
\end{example}

Support persists when a base is extended. This is what one expects of a notion of entitlement, and it simplifies the clauses that quantify over extensions.

\begin{lemma}[Monotonicity of support]\label{lem:mono-supp}
If $\supp_{\base{B}} \phi$ and $\base{C}$ extends $\base{B}$, then $\supp_{\base{C}} \phi$. Consequently:
\begin{enumerate}[label=(\roman*)]
  \item for every finite $\Delta$, possibly empty, $\Delta \supp_{\base{B}} \phi$ iff for every $\base{C}$ extending $\base{B}$, if $\base{C}$ supports every member of $\Delta$ then $\supp_{\base{C}} \phi$;
  \item $\Gamma \supp \phi$ iff $\Gamma \supp_{\emptyset} \phi$.
\end{enumerate}
\end{lemma}
\begin{proof}
The first claim is proved as Sandqvist's Lemma~3.2~\cite{Sandqvist2015}, by induction on the degree of $\phi$.
\begin{itemize}[label=--]
  \item $\phi = p$ or $\phi = \bot$. By Lemma~\ref{lem:mono}, every basic sentence derivable in $\base{B}$ is derivable in $\base{C}$.
  \item $\phi = \phi_1 \land \phi_2$ or $\phi = \forall x\phi'$. By the induction hypothesis, applied to the conjuncts or to the instances.
  \item $\phi = \phi_1 \lor \phi_2$, $\phi = \phi_1 \to \phi_2$, or $\phi = \exists x\phi'$. The clause quantifies over all extensions of $\base{B}$, and every extension of $\base{C}$ is one.
\end{itemize}
Part~(i) is (Inf) when $\Delta$ is nonempty, and the first claim when $\Delta$ is empty. Part~(ii) follows from~(i), since every base extends $\emptyset$.
\end{proof}

The next property says that support does not depend on which names are used, only on how they are related to the base. A renaming is \emph{bijective} if it is a permutation of $\setIndCons$.

\begin{lemma}[Equivariance of support]\label{lem:equiv-supp}
Let $\pi$ be a bijective renaming. Then $\Delta \supp_{\base{B}} \phi$ iff $\pi\Delta \supp_{\pi\base{B}} \pi\phi$.
\end{lemma}
\begin{proof}
By induction on the sum of the degrees of $\Delta$ and $\phi$. Since $\pi^{-1}$ is also a bijective renaming, it suffices to prove one direction.
\begin{itemize}[label=--]
  \item \emph{Basic sentences and $\bot$.} Apply Lemma~\ref{lem:renaming} to $\pi$. This covers $\bot$ because $\pi$ permutes the basic sentences.
  \item \emph{Clauses that quantify over extensions.} A base $\base{C}$ extends $\base{B}$ iff $\pi\base{C}$ extends $\pi\base{B}$, and every base is of the form $\pi\base{C}$. So the extensions of $\pi\base{B}$ are the images of the extensions of $\base{B}$, and the induction hypothesis applies to each.
  \item \emph{Quantifiers.} We have $(\pi\phi)[x \mapsto \pi a] = \pi(\phi[x \mapsto a])$, and $\pi$ permutes the constants. So the instances of $\pi(\forall x\phi)$ are the images of the instances of $\forall x\phi$, and likewise for $\exists x\phi$.
  \item \emph{Conjunction.} By the induction hypothesis. \qedhere
\end{itemize}
\end{proof}

Equivariance shows that the substitutional clause for $\forall$ needs to consult only finitely many instances.

\begin{corollary}[Finite instances]\label{cor:fresh-instance}
$\supp_{\base{B}} \forall x\phi$ iff $\base{B}$ supports $\phi[x \mapsto a]$ for every $a \in \con{\base{B}} \cup \con{\forall x \phi}$ and for one constant $e$ outside that set.
\end{corollary}
\begin{proof}
Left to right is immediate from $(\forall)$. For the converse, let $a$ be a constant outside $\con{\base{B}} \cup \con{\forall x \phi}$. The bijective renaming that exchanges $a$ and $e$ leaves $\base{B}$ and $\forall x \phi$ unchanged and sends $\phi[x \mapsto e]$ to $\phi[x \mapsto a]$. So Lemma~\ref{lem:equiv-supp} transfers support from the one to the other.
\end{proof}

So the substitutional clause never requires an enumeration. Beyond the finitely many names that the base or the formula mentions, a universal holds, if at all, in virtue of the base's general content (Section~\ref{sec:discussion}).

The final property of this section is the key lemma of Sandqvist's method~\cite[Lemma~2.2 and Theorem~3.1]{Sandqvist2015}. It says that, for basic sentences, the hypothetical judgement given by (Inf) coincides with derivability from assumptions. Deriving $u$ from assumptions $\atset{T}$ is the same as deriving $u$ in every extension of the base that makes the members of $\atset{T}$ derivable. In other words, adding an axiom to a base has the same effect as assuming its conclusion: the rules of a base behave like assumptions.

The lemma is where part~\ref{G3} of the freshness condition is needed. Its proof converts axioms into assumptions, and an eigenconstant must remain fresh for those assumptions. That holds only because freshness for the base was required to begin with.

\begin{lemma}[Conservativity]\label{lem:conservative}
For every base $\base{B}$, every finite set $\atset{T}$ of basic sentences, and every basic sentence $u$: $\atset{T} \proves_{\base{B}} u$ iff for every base $\base{C}$ extending $\base{B}$, if $\proves_{\base{C}} t$ for every $t \in \atset{T}$, then $\proves_{\base{C}} u$. That is, $\atset{T} \supp_{\base{B}} u$ iff $\atset{T} \proves_{\base{B}} u$.
\end{lemma}
\begin{proof}
\emph{Left to right.} This is Sandqvist's grafting argument~\cite[Lemma~2.2]{Sandqvist2015}. By Lemma~\ref{lem:mono}, a derivation of $u$ from $\atset{T}$ legal in $\base{B}$ is legal in any extension $\base{C}$. Suppose $\base{C}$ derives every member of $\atset{T}$. Graft onto each open leaf of the derivation a derivation in $\base{C}$, without open assumptions, of the basic sentence at that leaf. Grafting only removes open assumptions, so~(G) is preserved, and the result is a derivation of $u$ in $\base{C}$ without open assumptions.

\emph{Right to left.} Let $\base{C}$ be $\base{B}$ together with the axioms $\infer{t}{}$ for $t \in \atset{T}$. This is a base, since $\con{\base{C}} = \con{\base{B}} \cup \con{\atset{T}}$ is finite, and it derives every member of $\atset{T}$. By hypothesis there is a derivation $\mathcal{D}$ of $u$, legal in $\base{C}$, without open assumptions. Replace every application of one of the added axioms by the corresponding leaf $t$. The result applies only rules of $\base{B}$ and has its open assumptions in $\atset{T}$.

The result is legal in $\base{B}$. An eigenconstant $e$ at some application was fresh for $\base{C}$. It is therefore fresh for $\base{B}$ and absent from $\atset{T}$. So the newly open assumptions do not mention $e$, and condition~(G) still holds.
\end{proof}

Part~\ref{G3} does its work in the right-to-left direction. The added axioms for $\atset{T}$ belong to $\base{C}$, so an eigenconstant fresh for $\base{C}$ is automatically absent from $\atset{T}$. When the axioms become assumptions, the eigenvariable condition for those assumptions holds without further argument.

Without~\ref{G3}, the lemma would fail. Let $\base{B}$ be the base $\base{B}_2$ of Example~\ref{ex:fresh} without its axiom $H(s)$. Every extension of $\base{B}$ that derives $H(s)$ would then derive $U$, by the derivation displayed in the example, with eigenconstant $s$. But $U$ is not derivable in $\base{B}$ from the assumption $H(s)$, since an open assumption may not mention the eigenconstant.

\section{Soundness}\label{sec:soundness}

Support was defined without reference to any proof system for the logical constants. We now relate it to a familiar one. Let $\NI$ be Gentzen and Prawitz's natural deduction for intuitionistic first-order logic~\cite{Gentzen1969,Prawitz2006natural}, given in Figure~\ref{fig:NI}. We work with closed formulae throughout and use individual constants as eigenvariables. This is equivalent to the usual formulation with free variables. For a finite set $\Gamma$ of closed formulae and a closed formula $\phi$, we write $\Gamma \Ndash \phi$ if there is an $\NI$-derivation of $\phi$ all of whose open assumptions belong to $\Gamma$.

\begin{figure}[t]
\hrule\medskip
\[
\begin{array}{c@{\qquad}c@{\qquad}c}
\infer[\irn\land]{\phi \land \psi}{\phi & \psi}
&
\infer[\ern{\land_1}]{\phi}{\phi \land \psi}
&
\infer[\ern{\land_2}]{\psi}{\phi \land \psi}
\\[4mm]
\infer[\irn{\lor_1}]{\phi \lor \psi}{\phi}
\quad
\infer[\irn{\lor_2}]{\phi \lor \psi}{\psi}
&
\multicolumn{2}{c}{\infer[\ern\lor]{\chi}{\phi \lor \psi & \deduce{\chi}{[\phi]} & \deduce{\chi}{[\psi]}}}
\\[4mm]
\infer[\irn\to]{\phi \to \psi}{\deduce{\psi}{[\phi]}}
&
\infer[\ern\to]{\psi}{\phi \to \psi & \phi}
&
\infer[\efq]{\chi}{\bot}
\\[4mm]
\infer[\irn\forall]{\forall x\phi}{\phi[x \mapsto e]}
&
\infer[\ern\forall]{\phi[x \mapsto a]}{\forall x\phi}
&
\\[4mm]
\infer[\irn\exists]{\exists x\phi}{\phi[x \mapsto a]}
&
\multicolumn{2}{c}{\infer[\ern\exists]{\chi}{\exists x\phi & \deduce{\chi}{[\phi[x \mapsto e]]}}}
\end{array}
\]
\medskip
In $\irn\forall$, the eigenconstant $e$ occurs neither in $\forall x\phi$ nor in any open assumption on which the premiss rests. In $\ern\exists$, $e$ occurs neither in $\exists x\phi$, nor in $\chi$, nor in any open assumption of the minor premiss other than the discharged $\phi[x \mapsto e]$.
\medskip\hrule
\caption{Natural deduction}\label{fig:NI}
\end{figure}

\begin{theorem}[Soundness]\label{thm:soundness}
If $\Gamma \Ndash \phi$, then $\Gamma \supp \phi$.
\end{theorem}

The proof proceeds by induction on derivations, showing that each rule preserves validity. Two points need preparation: the eigenconstant of $\irn\forall$, and the rules that conclude an arbitrary formula.

The first concerns the rule $\irn\forall$, which infers $\forall x\phi$ from $\phi[x \mapsto e]$. To show that $\forall x\phi$ is supported we need every instance $\phi[x \mapsto a]$, but knowing that the premiss is valid tells us only about the eigenconstant $e$. We therefore prove something stronger: the sequent established by a derivation remains valid under every renaming of its constants. Then the premiss of $\irn\forall$ is valid with $a$ renamed for $e$, and since $e$ occurs nowhere else, this is exactly the instance we need.

\begin{definition}[Stable validity]\label{def:stable}
Let $\Delta$ be a finite set of closed formulae and $\phi$ a closed formula. The pair $(\Delta : \phi)$ is \emph{stably valid} if $\Delta\tau \supp_{\base{B}} \phi\tau$ for every base $\base{B}$ and every renaming $\tau$.
\end{definition}

By Lemma~\ref{lem:mono-supp}(i), $(\Delta : \phi)$ is stably valid iff, for every renaming $\tau$ and every base $\base{C}$ that supports every member of $\Delta\tau$, $\base{C}$ supports $\phi\tau$. This is the form in which we use it.

The second point is that three of the rules, $\efq$, $\ern\lor$, and $\ern\exists$, conclude an arbitrary formula $\chi$, whereas the clauses $(\bot)$, $(\lor)$, and $(\exists)$ speak only of basic sentences. The following lemma bridges the gap. It is Sandqvist's argument for disjunction~\cite[proof of Thm.~4.1]{Sandqvist2015}, extended to absurdity and to the quantifiers.

\begin{lemma}[Elimination into arbitrary formulae]\label{lem:elim}
For every base $\base{B}$ and closed formulae $\phi$, $\psi$, $\chi$:
\begin{enumerate}[label=(\alph*)]
  \item if $\supp_{\base{B}} \bot$, then $\supp_{\base{B}} \chi$;
  \item if $\supp_{\base{B}} \phi \lor \psi$, $\phi \supp_{\base{B}} \chi$, and $\psi \supp_{\base{B}} \chi$, then $\supp_{\base{B}} \chi$;
  \item if $\supp_{\base{B}} \exists x\phi$ and $\phi[x \mapsto a] \supp_{\base{B}} \chi$ for every constant $a$, then $\supp_{\base{B}} \chi$.
\end{enumerate}
\end{lemma}

The three parts are instances of one claim, because the clauses $(\bot)$, $(\lor)$, and $(\exists)$ have one form. Each says that whatever basic sentence follows from every member of a family of formulae follows outright. The family is empty for $\bot$, it is $\{\phi, \psi\}$ for $\phi \lor \psi$, and it consists of the instances $\phi[x \mapsto a]$ for $\exists x\phi$. The claim extends this from basic sentences to arbitrary formulae $\chi$, by pushing $\chi$ down to the basic sentences from which it is built.

\begin{proof}
Let $(\phi_i)_{i \in I}$ be a family of closed formulae. Say that a base $\base{B}$ \emph{eliminates} the family if, for every base $\base{C}$ extending $\base{B}$ and every basic sentence $p$,
\[
\text{if } \phi_i \supp_{\base{C}} p \text{ for every } i \in I, \text{ then } \supp_{\base{C}} p .
\]
We claim: \emph{if $\base{B}$ eliminates the family and $\phi_i \supp_{\base{B}} \chi$ for every $i \in I$, then $\supp_{\base{B}} \chi$.}

The claim gives the lemma. For (b), the clause $(\lor)$ says that $\base{B}$ eliminates $\{\phi, \psi\}$. For (c), the clause $(\exists)$ says that $\base{B}$ eliminates the family of instances of $\phi$. For (a), if $\supp_{\base{B}} \bot$ then every extension of $\base{B}$ derives every basic sentence, by $(\bot)$ and Lemma~\ref{lem:mono-supp}. So $\base{B}$ eliminates the empty family, and the second hypothesis of the claim is vacuous.

We prove the claim by induction on the degree of $\chi$, for all bases at once. Both hypotheses persist: if they hold at $\base{B}$, they hold at every base extending $\base{B}$, by Lemma~\ref{lem:mono-supp}.
\begin{itemize}[label=--]
  \item $\chi = p$. Since $\base{B}$ eliminates the family, take $\base{C} := \base{B}$.
  \item $\chi = \bot$. Let $p$ be a basic sentence. Every base supporting $\bot$ supports $p$, so $\phi_i \supp_{\base{B}} p$ for every $i$. The previous case gives $\supp_{\base{B}} p$. As $p$ was arbitrary, $\supp_{\base{B}} \bot$.
  \item $\chi = \chi_1 \land \chi_2$. Every base supporting $\chi$ supports $\chi_1$ and $\chi_2$, so $\phi_i \supp_{\base{B}} \chi_j$ for every $i$ and for $j = 1, 2$. The induction hypothesis gives $\supp_{\base{B}} \chi_1$ and $\supp_{\base{B}} \chi_2$.
  \item $\chi = \chi_1 \to \chi_2$. Let $\base{C}$ extend $\base{B}$ and support $\chi_1$. We show $\supp_{\base{C}} \chi_2$. If $\base{D}$ extends $\base{C}$ and supports $\phi_i$, then $\base{D}$ supports $\chi$ by hypothesis and $\chi_1$ by monotonicity, hence $\chi_2$. So $\phi_i \supp_{\base{C}} \chi_2$ for every $i$. The induction hypothesis, applied to $\chi_2$ at the base $\base{C}$, gives $\supp_{\base{C}} \chi_2$.
  \item $\chi = \chi_1 \lor \chi_2$. Let $\base{C}$ extend $\base{B}$, and let $p$ be a basic sentence with $\chi_1 \supp_{\base{C}} p$ and $\chi_2 \supp_{\base{C}} p$. We show $\supp_{\base{C}} p$. If $\base{D}$ extends $\base{C}$ and supports $\phi_i$, then $\base{D}$ supports $\chi$, and the clause $(\lor)$ for $\chi$ at $\base{D}$ gives $\supp_{\base{D}} p$. So $\phi_i \supp_{\base{C}} p$ for every $i$. Since $\base{B}$ eliminates the family, $\supp_{\base{C}} p$.
  \item $\chi = \forall y\chi'$. Every base supporting $\chi$ supports each instance $\chi'[y \mapsto b]$, so $\phi_i \supp_{\base{B}} \chi'[y \mapsto b]$ for every $i$. The instance has smaller degree, and the induction hypothesis gives $\supp_{\base{B}} \chi'[y \mapsto b]$ for every $b$.
  \item $\chi = \exists y\chi'$. As for $\chi_1 \lor \chi_2$, with the instances $\chi'[y \mapsto b]$ in place of $\chi_1$ and $\chi_2$ and the clause $(\exists)$ in place of $(\lor)$. \qedhere
\end{itemize}
\end{proof}

The proof extends a clause from basic sentences to all formulae by an induction on the formula, and we call an argument of this kind an \emph{overflow} argument. It succeeds here because the family $(\phi_i)$ and the conclusion $\chi$ are independent of one another, so that $\chi$ can be taken apart while the family is held fixed. No such argument is available for the second-order quantifier (Remark~\ref{rem:overflow}), which is why the semantics of second-order logic needs comprehensive bases.

We can now prove the theorem.

\begin{proof}[Proof of Theorem~\ref{thm:soundness}]
We show, by induction on an $\NI$-derivation $\mathcal{D}$, that $(\Gamma : \phi)$ is stably valid, where $\phi$ is the conclusion of $\mathcal{D}$ and $\Gamma$ its set of open assumptions. The theorem follows by taking $\tau$ to be the identity, together with Lemma~\ref{lem:mono-supp}(i): if every base supporting the open assumptions supports $\phi$, then so does every base supporting a larger set.

In each case, let $\Gamma$ be the open assumptions of $\mathcal{D}$, fix a renaming $\tau$, and fix a base $\base{C}$ supporting every member of $\Gamma\tau$. We show that $\base{C}$ supports the renamed conclusion. Renamings commute with the connectives and quantifiers, and $(\phi[x \mapsto a])\tau = (\phi\tau)[x \mapsto \tau a]$. The open assumptions of each premiss derivation are among $\Gamma$ together with the hypotheses discharged at that premiss. Hence, by the induction hypothesis, every extension of $\base{C}$ that supports the renamed hypotheses discharged at a premiss supports the renamed premiss. In particular, $\base{C}$ supports each renamed premiss at which nothing is discharged.

If $\mathcal{D}$ is a single assumption, the claim is trivial. The propositional rules follow Sandqvist's proof of soundness~\cite[Thm.~4.1]{Sandqvist2015}, applied to the renamed formulae.

\begin{itemize}[label=--]
  \item $\irn\land$, $\ern{\land_1}$, $\ern{\land_2}$. By the clause $(\land)$.
  \item $\irn\to$. Every extension of $\base{C}$ that supports $\phi\tau$ supports $\psi\tau$. This is $\phi\tau \supp_{\base{C}} \psi\tau$, so $\base{C}$ supports $(\phi \to \psi)\tau$.
  \item $\ern\to$. The base $\base{C}$ supports $(\phi \to \psi)\tau$ and $\phi\tau$. By $(\to)$ and (Inf), taken at $\base{C}$ itself, it supports $\psi\tau$.
  \item $\irn{\lor_1}$, $\irn{\lor_2}$. Suppose $\base{C}$ supports $\phi\tau$. Let $\base{D}$ extend $\base{C}$ and let $p$ be such that $\phi\tau \supp_{\base{D}} p$ and $\psi\tau \supp_{\base{D}} p$. Since $\base{D}$ supports $\phi\tau$, it supports $p$. So $\base{C}$ supports $(\phi \lor \psi)\tau$. The other rule is symmetric.
  \item $\ern\lor$. The base $\base{C}$ supports $(\phi \lor \psi)\tau$, and $\phi\tau \supp_{\base{C}} \chi\tau$ and $\psi\tau \supp_{\base{C}} \chi\tau$. Lemma~\ref{lem:elim}(b) gives $\supp_{\base{C}} \chi\tau$.
  \item $\efq$. The base $\base{C}$ supports $\bot$. Lemma~\ref{lem:elim}(a) gives $\supp_{\base{C}} \chi\tau$.
\end{itemize}
The quantifier rules are as follows.
\begin{itemize}[label=--]
  \item $\ern\forall$. Immediate from $(\forall)$.
  \item $\irn\exists$. Suppose $\base{C}$ supports $(\phi\tau)[x \mapsto \tau a]$. Let $\base{D}$ extend $\base{C}$ and let $p$ be such that $(\phi\tau)[x \mapsto b] \supp_{\base{D}} p$ for every $b$. Taking $b := \tau a$, and since $\base{D}$ supports $(\phi\tau)[x \mapsto \tau a]$, it supports $p$. So $\base{C}$ supports $(\exists x\phi)\tau$.
  \item $\irn\forall$. Let the premiss be $\phi[x \mapsto e]$, resting on assumptions $\Gamma$, with $e$ in neither $\Gamma$ nor $\forall x\phi$. Let $a$ be a constant, and let $\tau'$ agree with $\tau$ except that $\tau'(e) = a$. Since $e$ occurs neither in $\Gamma$ nor in $\phi$, we have $\Gamma\tau' = \Gamma\tau$ and $(\phi[x \mapsto e])\tau' = (\phi\tau)[x \mapsto a]$. Stable validity of the premiss, applied to $\tau'$, shows that $\base{C}$ supports $(\phi\tau)[x \mapsto a]$. As $a$ was arbitrary, $\base{C}$ supports $\forall x(\phi\tau) = (\forall x\phi)\tau$.
  \item $\ern\exists$. Let $\Gamma'$ be the open assumptions of the minor premiss other than the discharged $\phi[x \mapsto e]$, so that $e$ occurs in none of $\Gamma'$, $\exists x\phi$, and $\chi$. The major premiss gives $\supp_{\base{C}} \exists x(\phi\tau)$. For each constant $a$, let $\tau'$ agree with $\tau$ except that $\tau'(e) = a$. As in the previous case, stable validity of the minor premiss, applied to $\tau'$, shows that every extension of $\base{C}$ supporting $(\phi\tau)[x \mapsto a]$ supports $\chi\tau$. So $(\phi\tau)[x \mapsto a] \supp_{\base{C}} \chi\tau$ for every $a$, and Lemma~\ref{lem:elim}(c) gives $\supp_{\base{C}} \chi\tau$. \qedhere
\end{itemize}
\end{proof}

\begin{remark}[Finite bases]\label{rem:soundness-finite}
The proof uses nothing about bases beyond Lemma~\ref{lem:mono-supp}. It therefore holds for any class of bases closed under adding rules, and in particular if bases are required to be finite.
\end{remark}

\section{Completeness}\label{sec:completeness}

This section proves the converse of Theorem~\ref{thm:soundness}, and with it the first main result.

\begin{theorem}[Completeness]\label{thm:A}
Let $\Gamma$ be a finite set of closed first-order formulae and $\zeta$ a closed first-order formula. Then $\Gamma \supp \zeta$ iff $\Gamma \Ndash \zeta$. Moreover, this remains true if bases are required to be finite.
\end{theorem}

\emph{Idea.} We follow the method of Sandqvist~\cite{Sandqvist2015}, which needs no canonical models. It simulates natural deduction by atomic sentences in a base tailored to the inference at hand. To each formula $\gamma$ that may occur in a derivation of $\zeta$ from $\Gamma$ we assign a basic sentence $\flatop{\gamma}$, its \emph{counterpart}. We then build a base $\base{N}$ whose rules make each counterpart behave, inferentially, exactly as the formula it stands for behaves in natural deduction. For example, if $\gamma$ is $A \land B$ with $A$ and $B$ basic, $\base{N}$ contains the rules
\[
\infer{\flatop{\gamma}}{A & B} \qquad \infer{A}{\flatop{\gamma}} \qquad \infer{B}{\flatop{\gamma}}
\]

The argument then has three steps.
\begin{enumerate}[label=(\arabic*)]
  \item \emph{Simulation.} In $\base{N}$ and its extensions, support of a formula coincides with derivability of its counterpart (Lemma~\ref{lem:dagger}).
  \item \emph{Translation.} A derivation in $\base{N}$ is a natural-deduction derivation in disguise: replacing each counterpart by the formula it stands for turns it into an $\NI$-derivation (Lemma~\ref{lem:ddagger}).
  \item \emph{Completeness.} If $\Gamma \supp \zeta$, then in particular $\base{N}$ supports $\zeta$ given $\Gamma$. By (1) and Lemma~\ref{lem:conservative}, the counterpart of $\zeta$ is derivable in $\base{N}$ from those of $\Gamma$. By (2), $\Gamma \Ndash \zeta$.
\end{enumerate}

For propositional logic this is Sandqvist's proof. The quantifiers raise two new problems, and this section is mostly about them. The first is that the relevant formulae include all the instances $\phi[x \mapsto a]$ of the quantified formulae. There are infinitely many of them, while a base may mention only finitely many names. The second is that the eigenvariable conditions of $\irn\forall$ and $\ern\exists$ must be simulated by the base. The rules with constraints of Section~\ref{sec:bases} answer the second problem. The choice of counterparts answers the first.

Throughout this section we fix a finite set $\Gamma$ of closed formulae and a closed formula $\zeta$.

We begin with counterparts. Sandqvist takes the counterpart of each formula to be a fresh propositional atom. That will not do here, for a reason that the eigenvariable condition makes plain. Suppose a derivation in the base applies the counterpart of $\irn\forall$, concluding $\flatop{(\forall x\phi)}$ from $\flatop{(\phi[x \mapsto e])}$ with eigenconstant $e$. Part~\ref{G1} of the freshness condition guarantees that $e$ does not occur in the basic sentences on which the premiss rests. For the translated derivation to be correct, we need $e$ not to occur in the \emph{formulae} those basic sentences stand for.

This holds if each counterpart contains exactly the constants of the formula it stands for. So the counterpart of $\phi[x \mapsto a]$ must be an atom with $a$ among its arguments. One predicate constant then serves for all formulae of the same form, and the base needs only finitely many rules.

To make `the same form' precise, we abstract constants. The \emph{shape} $\hat\gamma$ of a formula $\gamma$ is obtained by replacing each \emph{occurrence} of a constant or of a free variable by a distinct new variable, in left-to-right order. We write $\vec s_\gamma$ for the list of the terms replaced, occurrence by occurrence and with repetitions, so that $\gamma$ is $\hat\gamma$ with $\vec s_\gamma$ put back. For example, the shape of $R(a,a) \land Q(b)$ is $R(u_1, u_2) \land Q(u_3)$, and its list of terms is $a, a, b$. Bound variables are not abstracted: the shape of $\forall x \phi$ is $\forall x$ followed by $\phi$ with every term occurrence abstracted except those of $x$.

It matters that shapes abstract occurrences and not constants, because this makes the shape of an instance $\phi[x \mapsto a]$ independent of $a$: the formulae $P(a) \to P(a)$, $P(b) \to P(b)$, and $P(a) \to P(b)$ all have the shape $P(u_1) \to P(u_2)$. Had we abstracted constants instead, the instance of $\forall x(P(x) \to P(c))$ at $c$ would differ in form from its instance at any other name, and no single rule could pass from the universal to its instances.

Let $S$ be the set of shapes of the subformulae of $\Gamma \cup \{\zeta\}$, where subformulae are taken with their free variables. The set $S$ is finite. Let $\Theta$ be the set of closed formulae whose shape lies in $S$. It contains $\Gamma \cup \{\zeta\}$, it is closed under taking the immediate subformulae of $\land$, $\lor$, and $\to$, and it contains every instance $\phi[x \mapsto a]$ of each $\forall x\phi$ and $\exists x \phi$ it contains. The set $\Theta$ plays the role of Sandqvist's set of subformulae of $\Gamma$ and $\zeta$. It is infinite, but its members fall into the finitely many shapes of $S$.

For each shape in $S$ that is not atomic, including the shape $\bot$, pick a predicate constant $Q_{\hat\gamma}$ of arity the length of $\vec s_\gamma$, not occurring in $\Gamma \cup \{\zeta\}$, distinct for distinct shapes.

\begin{definition}[Counterpart]\label{def:counterpart}
For $\gamma \in \Theta$, the \emph{counterpart} of $\gamma$ is
\[
\flatop{\gamma} := \gamma \quad\text{if $\gamma$ is atomic,} \qquad\qquad \flatop{\gamma} := Q_{\hat\gamma}(\vec s_\gamma) \quad\text{otherwise.}
\]
\end{definition}

For example, if $R(a,a) \land Q(b) \in \Theta$, its counterpart is $C(a,a,b)$, where $C$ is the predicate constant chosen for the shape $R(u_1, u_2) \land Q(u_3)$. The counterpart of $\bot$ is a $0$-ary atom.

\begin{lemma}[Counterparts]\label{lem:flat}
The map $\gamma \mapsto \flatop{\gamma}$ from $\Theta$ to basic sentences is injective, and $\con{\flatop{\gamma}} = \con{\gamma}$ for every $\gamma \in \Theta$.
\end{lemma}
\begin{proof}
Formulae of distinct nonatomic shapes receive distinct predicate constants, and a formula of a given shape is determined by its list of terms. A counterpart of a nonatomic formula cannot coincide with an atomic formula of $\Theta$: the predicate constants of atoms in $\Theta$ occur in $\Gamma \cup \{\zeta\}$, and the $Q_{\hat\gamma}$ do not. The second claim holds because $\vec s_\gamma$ lists exactly the constants of $\gamma$.
\end{proof}

We write $\natop{(-)}$ for the inverse of counterparts: $\natop{p}$ is the formula of which $p$ is the counterpart, if there is one, and $p$ itself otherwise. This is defined on all basic sentences. By Lemma~\ref{lem:flat} it is injective, preserves constants, and satisfies $\natop{(\flatop{\gamma})} = \gamma$.

We can now define the simulation base. Its rules are written with shapes, so that a single rule covers all formulae of a given shape. Extend counterparts to shapes by putting $\flatop{\hat\gamma} := \hat\gamma$ for atomic $\hat\gamma$ and $\flatop{\hat\gamma} := Q_{\hat\gamma}(\vec u)$ otherwise, where $\vec u$ are the variables of $\hat\gamma$. Then $\flatop{\gamma}$ is $\flatop{\hat\gamma}$ with $\vec s_\gamma$ put for $\vec u$.

\begin{definition}[Simulation base]\label{def:N}
The \emph{simulation base} $\base{N}$ consists of the rules of Figure~\ref{fig:N}, one group for each nonatomic shape in $S$.
\end{definition}

\begin{figure}[t]
\hrule\medskip
\[
\begin{array}{c@{\qquad}c@{\qquad}c}
\infer[\irn\land]{\flatop{(\phi \land \psi)}}{\flatop{\phi} & \flatop{\psi}}
&
\infer[\ern{\land_1}]{\flatop{\phi}}{\flatop{(\phi \land \psi)}}
&
\infer[\ern{\land_2}]{\flatop{\psi}}{\flatop{(\phi \land \psi)}}
\\[4mm]
\infer[\irn{\lor_1}]{\flatop{(\phi \lor \psi)}}{\flatop{\phi}}
\quad
\infer[\irn{\lor_2}]{\flatop{(\phi \lor \psi)}}{\flatop{\psi}}
&
\multicolumn{2}{c}{\infer[\ern\lor]{Z}{\flatop{(\phi \lor \psi)} & \deduce{Z}{[\flatop{\phi}]} & \deduce{Z}{[\flatop{\psi}]}}}
\\[4mm]
\infer[\irn\to]{\flatop{(\phi \to \psi)}}{\deduce{\flatop{\psi}}{[\flatop{\phi}]}}
&
\infer[\ern\to]{\flatop{\psi}}{\flatop{(\phi \to \psi)} & \flatop{\phi}}
&
\infer[\efq]{Z}{\flatop{\bot}}
\\[4mm]
\infer[\irn\forall]{\flatop{(\forall x\phi)}}{\flatop{(\phi[x \mapsto y])}}\ \fresh{y}
&
\infer[\ern\forall]{\flatop{(\phi[x \mapsto y])}}{\flatop{(\forall x\phi)}}
&
\\[4mm]
\infer[\irn\exists]{\flatop{(\exists x\phi)}}{\flatop{(\phi[x \mapsto y])}}
&
\multicolumn{2}{c}{\infer[\ern\exists]{Z}{\flatop{(\exists x\phi)} & \deduce{Z}{[\flatop{(\phi[x \mapsto y])}]}}\ \fresh{y}}
\end{array}
\]
\medskip
Here $\phi$ and $\psi$ stand for shapes, so that their variables are variables of the rule; $Z$ is an atom variable, so that $\ern\lor$, $\ern\exists$, and $\efq$ may conclude any basic sentence; and $\flatop{(\phi[x \mapsto y])}$ is the counterpart of the shape of the instance, with $y$ in every argument place that came from an occurrence of $x$.
\medskip\hrule
\caption{The simulation base}\label{fig:N}
\end{figure}

Each rule of $\base{N}$ is the counterpart of the rule of $\NI$ with the same name, and the constraints on $\irn\forall$ and $\ern\exists$ are the eigenvariable conditions of those rules. The base $\base{N}$ is finite, because $S$ is, and it mentions no individual constants at all, since shapes contain none.

The instances of the rules of $\base{N}$ are of two kinds. Most are the counterparts of the instances of the rules of $\NI$ for formulae in $\Theta$. The others are the instances of $\ern\lor$, $\ern\exists$, and $\efq$ whose conclusion $Z$ is a basic sentence that is not a counterpart.

\begin{example}[A simulation base]\label{ex:N}
Let $\Gamma$ be empty and $\zeta := \forall x(P(x) \to P(x))$. The subformulae of $\zeta$ have the shapes
\[
\forall x(P(x) \to P(x)), \qquad P(u_1) \to P(u_2), \qquad P(u).
\]
Write $K$ for the $0$-ary predicate constant chosen for the first shape and $I$ for the binary one chosen for the second. Thus $\flatop{\zeta} = K$, and the counterpart of $P(a) \to P(b)$ is $I(a,b)$. The base $\base{N}$ consists of four rules:
\[
\infer[\irn\to]{I(u_1, u_2)}{\deduce{P(u_2)}{[P(u_1)]}}
\qquad
\infer[\ern\to]{P(u_2)}{I(u_1, u_2) & P(u_1)}
\qquad
\infer[\irn\forall]{K}{I(y, y)}\ \fresh{y}
\qquad
\infer[\ern\forall]{I(y, y)}{K}
\]
The instance $P(y) \to P(y)$ of the universal has the counterpart $I(y,y)$, because both argument places come from occurrences of the bound variable $x$.

Since $\base{N}$ mentions no constants, every constant $e$ is fresh for it. The derivation on the left below is therefore legal in $\base{N}$: the premiss of $\irn\forall$ rests on no open assumption, since $P(e)$ is discharged. It derives $K$, the counterpart of $\zeta$. Replacing each counterpart by the formula it stands for turns it into the $\NI$-derivation on the right:
\[
\infer[\irn\forall]{K}{\infer[\irn\to^{1}]{I(e,e)}{[P(e)]^{1}}}
\qquad\qquad\qquad
\infer[\irn\forall]{\forall x(P(x) \to P(x))}{\infer[\irn\to^{1}]{P(e) \to P(e)}{[P(e)]^{1}}}
\]
The constraint of the base is the eigenvariable condition of $\NI$. Lemmas~\ref{lem:dagger} and~\ref{lem:ddagger} show that this correspondence is general.
\end{example}

The proof rests on three lemmas. The first is Lemma~\ref{lem:conservative}: for basic sentences, derivability from assumptions coincides with the judgement (Inf). The other two are the steps (1) and (2) of the plan above. We state them here, derive the theorem, and prove them at the end of the section.

The simulation lemma says that in $\base{N}$, and in every base that contains it, a formula is supported exactly when its counterpart is derivable.

\begin{lemma}[Simulation]\label{lem:dagger}
For every base $\base{B}$ extending $\base{N}$ and every $\gamma \in \Theta$: $\proves_{\base{B}} \flatop{\gamma}$ iff $\supp_{\base{B}} \gamma$.
\end{lemma}

The translation lemma says that whatever $\base{N}$ derives, $\NI$ derives once counterparts are replaced by the formulae they stand for.

\begin{lemma}[Translation]\label{lem:ddagger}
If $\mathcal{D}$ is a derivation legal in $\base{N}$, with conclusion $p$ and open assumptions $\atset{T}$, then $\natop{\atset{T}} \Ndash \natop{p}$.
\end{lemma}

\begin{proof}[Proof of Theorem~\ref{thm:A}]
Soundness is Theorem~\ref{thm:soundness}. For completeness, suppose $\Gamma \supp \zeta$. In particular, $\Gamma \supp_{\base{N}} \zeta$.

We claim that $\flatop{\Gamma} \proves_{\base{N}} \flatop{\zeta}$. If $\Gamma$ is empty, this is Lemma~\ref{lem:dagger} applied to $\base{N}$. Otherwise, by Lemma~\ref{lem:conservative} it suffices that every $\base{C}$ extending $\base{N}$ that derives $\flatop{\gamma}$ for each $\gamma \in \Gamma$ derives $\flatop{\zeta}$. By Lemma~\ref{lem:dagger} such a $\base{C}$ supports every member of $\Gamma$, hence supports $\zeta$, hence derives $\flatop{\zeta}$.

Now take a derivation witnessing $\flatop{\Gamma} \proves_{\base{N}} \flatop{\zeta}$, with open assumptions $\atset{T} \subseteq \flatop{\Gamma}$. Lemma~\ref{lem:ddagger} gives $\natop{\atset{T}} \Ndash \zeta$, and $\natop{\atset{T}} \subseteq \Gamma$. So $\Gamma \Ndash \zeta$.

For finite bases, note that $\base{N}$ is finite, that the base constructed in the proof of Lemma~\ref{lem:conservative} is finite whenever the given one is, and that soundness holds by Remark~\ref{rem:soundness-finite}. So the argument goes through verbatim when (Inf) quantifies over finite extensions only.
\end{proof}

It remains to prove the two lemmas. The simulation lemma is proved in two stages. The first shows that, in any extension of $\base{N}$, the counterparts obey the clauses of support, with derivability in place of support.

\begin{lemma}[Atomic clauses]\label{lem:atomic-clauses}
For every base $\base{B}$ extending $\base{N}$, and all formulae of $\Theta$ of the forms indicated:
\begin{enumerate}[label=(\roman*)]
  \item $\proves_{\base{B}} \flatop{\bot}$ iff $\proves_{\base{B}} p$ for every basic sentence $p$;
  \item $\proves_{\base{B}} \flatop{(\phi \land \psi)}$ iff $\proves_{\base{B}} \flatop{\phi}$ and $\proves_{\base{B}} \flatop{\psi}$;
  \item $\proves_{\base{B}} \flatop{(\phi \lor \psi)}$ iff for every $\base{C}$ extending $\base{B}$ and every $p$, if $\flatop{\phi} \proves_{\base{C}} p$ and $\flatop{\psi} \proves_{\base{C}} p$, then $\proves_{\base{C}} p$;
  \item $\proves_{\base{B}} \flatop{(\phi \to \psi)}$ iff $\flatop{\phi} \proves_{\base{B}} \flatop{\psi}$;
  \item $\proves_{\base{B}} \flatop{(\forall x\phi)}$ iff $\proves_{\base{B}} \flatop{(\phi[x \mapsto a])}$ for every constant $a$;
  \item $\proves_{\base{B}} \flatop{(\exists x\phi)}$ iff for every $\base{C}$ extending $\base{B}$ and every $p$, if $\flatop{(\phi[x \mapsto a])} \proves_{\base{C}} p$ for every constant $a$, then $\proves_{\base{C}} p$.
\end{enumerate}
\end{lemma}

In each clause, one direction applies an elimination rule of $\base{N}$ and the other an introduction rule. Clauses (i)--(iv) are Sandqvist's~\cite[proof of Thm.~5.1]{Sandqvist2015}, and their rules carry no constraints. The new clauses are (v) and (vi), and there the work is to choose an eigenconstant satisfying~(G).

\begin{proof}
We take the clauses in turn.

(i) Left to right: apply $\efq$ with $Z := p$. Right to left: $\flatop{\bot}$ is a basic sentence.

(ii) Left to right by $\ern{\land_1}$ and $\ern{\land_2}$, right to left by $\irn\land$.

(iii) Left to right: let $\base{C}$ and $p$ be as stated. By Lemma~\ref{lem:mono}, $\proves_{\base{C}} \flatop{(\phi \lor \psi)}$. Apply $\ern\lor$ with $Z := p$ to this derivation and to the derivations of $p$ from $\flatop{\phi}$ and from $\flatop{\psi}$, discharging those assumptions. Right to left: take $\base{C} := \base{B}$ and $p := \flatop{(\phi \lor \psi)}$. Then $\flatop{\phi} \proves_{\base{B}} p$ by $\irn{\lor_1}$ and $\flatop{\psi} \proves_{\base{B}} p$ by $\irn{\lor_2}$.

(iv) Left to right: apply $\ern\to$ to the derivation of $\flatop{(\phi \to \psi)}$ and the one-node derivation $\flatop{\phi}$. Right to left: apply $\irn\to$, discharging $\flatop{\phi}$.

(v) Left to right by $\ern\forall$. For the converse, choose a constant $e$ fresh for $\base{B}$ and not in $\con{\forall x\phi}$. By hypothesis there is a derivation of $\flatop{(\phi[x \mapsto e])}$, legal in $\base{B}$ and without open assumptions. Apply $\irn\forall$ with eigenconstant $e$. Condition~(G) holds: the premiss rests on no assumptions; $e$ does not occur in the rule instance, whose constants are those of $\forall x\phi$ by Lemma~\ref{lem:flat}; and $e$ is fresh for $\base{B}$.

(vi) Left to right: suppose $\proves_{\base{B}} \flatop{(\exists x\phi)}$, and let $\base{C}$ and $p$ be as stated. Choose $e$ fresh for $\base{C}$ and not in $\con{\exists x\phi} \cup \con{p}$. By hypothesis there is a derivation of $p$ from $\flatop{(\phi[x \mapsto e])}$, legal in $\base{C}$, and by Lemma~\ref{lem:mono} a derivation of $\flatop{(\exists x\phi)}$ legal in $\base{C}$ without open assumptions. Apply $\ern\exists$ with $Z := p$ and eigenconstant $e$, discharging $\flatop{(\phi[x \mapsto e])}$. Condition~(G) holds: the minor premiss rests on no assumption other than the discharged one; $e$ does not occur in the rule instance, whose constants are those of $\exists x\phi$ and $p$; and $e$ is fresh for $\base{C}$. Right to left: take $\base{C} := \base{B}$ and $p := \flatop{(\exists x\phi)}$. Each $\flatop{(\phi[x \mapsto a])} \proves_{\base{B}} \flatop{(\exists x\phi)}$ by $\irn\exists$.
\end{proof}

The second stage compares Lemma~\ref{lem:atomic-clauses} with Figure~\ref{fig:support}. The two sets of clauses match line by line, with one difference: the right-hand sides of (iii), (iv), and (vi) speak of derivability from assumptions where the support clauses speak of (Inf). Lemma~\ref{lem:conservative} closes exactly this gap.

\begin{proof}[Proof of Lemma~\ref{lem:dagger}]
By induction on the degree of $\gamma$, for all bases extending $\base{N}$ at once. The induction hypothesis applies at every extension of $\base{B}$, since such extensions extend $\base{N}$. It yields two facts about formulae $\gamma', \gamma'' \in \Theta$ of smaller degree than $\gamma$, for every base $\base{C}$ extending $\base{B}$ and every basic sentence $p$:
\begin{equation}\label{eq:inf-flat}
\gamma' \supp_{\base{C}} \gamma'' \ \text{ iff }\ \flatop{\gamma'} \proves_{\base{C}} \flatop{\gamma''}
\qquad\text{and}\qquad
\gamma' \supp_{\base{C}} p \ \text{ iff }\ \flatop{\gamma'} \proves_{\base{C}} p .
\end{equation}
Indeed, by (Inf) the left-hand sides say that every extension of $\base{C}$ supporting $\gamma'$ supports $\gamma''$, or supports $p$. By the induction hypothesis and (At), this holds iff every extension of $\base{C}$ deriving $\flatop{\gamma'}$ derives $\flatop{\gamma''}$, or derives $p$. By Lemma~\ref{lem:conservative}, that is the right-hand side.
\begin{itemize}[label=--]
  \item $\gamma$ atomic. Then $\flatop{\gamma} = \gamma$ and the claim is (At).
  \item $\gamma = \bot$. By Lemma~\ref{lem:atomic-clauses}(i) and the clause $(\bot)$.
  \item $\gamma = \phi \land \psi$. By Lemma~\ref{lem:atomic-clauses}(ii), the induction hypothesis, and $(\land)$.
  \item $\gamma = \phi \to \psi$. By Lemma~\ref{lem:atomic-clauses}(iv), the first part of~\eqref{eq:inf-flat} with $\base{C} := \base{B}$, and $(\to)$.
  \item $\gamma = \phi \lor \psi$. By the second part of~\eqref{eq:inf-flat} and (At), the right-hand side of Lemma~\ref{lem:atomic-clauses}(iii) is the right-hand side of $(\lor)$.
  \item $\gamma = \forall x\phi$. Every instance $\phi[x \mapsto a]$ lies in $\Theta$ and has smaller degree than $\forall x\phi$. So Lemma~\ref{lem:atomic-clauses}(v), the induction hypothesis, and $(\forall)$ give the claim.
  \item $\gamma = \exists x\phi$. By the second part of~\eqref{eq:inf-flat}, applied to the instances, and (At), the right-hand side of Lemma~\ref{lem:atomic-clauses}(vi) is the right-hand side of $(\exists)$. \qedhere
\end{itemize}
\end{proof}

It remains to translate derivations in $\base{N}$ back into natural deduction. The translation replaces each basic sentence $q$ by $\natop{q}$, and each rule of Figure~\ref{fig:N} by the rule of $\NI$ with the same name. What must be checked is that discharges and eigenvariable conditions survive.

\begin{proof}[Proof of Lemma~\ref{lem:ddagger}]
By induction on $\mathcal{D}$. A one-node derivation $q$ becomes the assumption $\natop{q}$. Otherwise $\mathcal{D}$ ends with an instance of a rule of Figure~\ref{fig:N}, and the induction hypothesis translates the derivations of its premisses. As in Sandqvist's proof of~(\ddag)~\cite[Thm.~5.1]{Sandqvist2015}, the instance becomes an instance of the rule of $\NI$ with the same name, an instance $q$ of $Z$ becoming the formula $\natop{q}$.
\begin{itemize}[label=--]
  \item \emph{Rules that discharge hypotheses.} Since $\natop{(-)}$ is injective, the open assumptions of a translated premiss that equal the translation of a discharged hypothesis are exactly the translations of that hypothesis. So the rule of $\NI$ discharges what the rule of $\base{N}$ discharged, and no more.
  \item \emph{Vacuous quantifiers.} If the quantifier is vacuous, the rule of $\base{N}$ has no eigenconstant. We take for $e$ in the rule of $\NI$ a constant occurring nowhere in the translated derivation.
  \item $\irn\forall$ \emph{with eigenconstant} $e$. By~(G), $e$ occurs neither in the open assumptions of the premiss nor in $\flatop{(\forall x\phi)}$. Since $\natop{(-)}$ preserves constants, $e$ occurs neither in the translated assumptions nor in $\forall x\phi$.
  \item $\ern\exists$ \emph{with $Z := q$ and eigenconstant} $e$. By~(G), $e$ occurs neither in $\flatop{(\exists x \phi)}$, nor in $q$, nor in the open assumptions of the minor premiss other than the discharged one. So it occurs neither in $\exists x\phi$, nor in $\natop{q}$, nor in the translated open assumptions of the minor premiss other than $\phi[x \mapsto e]$. \qedhere
\end{itemize}
\end{proof}

The translation uses only~\ref{G1} and~\ref{G2}. Freshness for the base,~\ref{G3}, is vacuous in $\base{N}$, which mentions no constants. It is needed for the simulation lemma, through Lemma~\ref{lem:conservative}.

\section{Second-order Logic}\label{sec:second-order}

Theorem~\ref{thm:A} settles first-order quantification, where the quantifier ranges over names and a fresh name stands for an arbitrary object. Second-order quantification cannot be treated by the same means alone, because predicate names do not exhaust what may be said of objects. The target is the intuitionistic version of Henkin's second-order logic~\cite{henkin1950completeness}, in the natural-deduction formulation of van Dalen~\cite[Ch.~5]{van1994logic}, whose second-order quantifier may be instantiated by any formula.

This section extends the language and the bases, shows that the obvious substitutional clause is too weak, and gives the semantics that replaces it. Sections~\ref{sec:soundness-2} and~\ref{sec:completeness-2} prove soundness and completeness. Throughout, the work of Sections~\ref{sec:bases}--\ref{sec:completeness} is reused, and we say only what changes.

To the language we add, for each $n \geq 0$, a denumerable set of \emph{predicate variables} $X, Y, \ldots$ of arity $n$, together with the atomic formulae $X(t_1, \ldots, t_n)$ and the second-order universal quantifier $\foralltwo X\phi$. The second-order existential quantifier is not primitive; it is definable, as we show at the end of this section. The degree is extended by $\deg(\foralltwo X\phi) := \deg\phi + 1$, every atom having degree $1$.

We also need atomic sentences that say something \emph{about predicates}. For each finite list $\kappa$ of sorts, each sort being `individual' or `predicate of arity $n$', we fix a denumerable set $\HCONS_\kappa$ of \emph{higher-type predicate constants}. An atom $h(s_1, \ldots, s_k)$ with $h \in \HCONS_\kappa$ takes as arguments individual terms and predicate constants of the sorts listed in $\kappa$. For example, with $\mathsf{Br}$ for `is brave' and $t$ for Tammy, the atom $\mathsf{Ex}(\mathsf{Br}, t)$ might say that bravery is exemplified by Tammy.

The role of higher-type atoms here is technical, and becomes clear in Section~\ref{sec:completeness-2}: they supply the counterparts of second-order formulae. Two stipulations keep them in their place: the predicate arguments of a higher-type atom are always predicate \emph{constants}, never predicate variables, and no variable occupies the head position of such an atom, so there is no quantification over $\HCONS$. These stipulations concern formulae. Rules, and the shapes of Section~\ref{sec:completeness-2}, are metalinguistic schemata, and may have variables of predicate sort in argument positions.

The basic sentences are now all closed atoms, ordinary or higher-type. For a formula, rule, or set of these, $\con{-}$ is now the set of individual \emph{and predicate} constants occurring in it; higher-type constants are not counted.

To say what may be substituted for a predicate variable we use $\lambda$-notation, as a device of the metalanguage only. If $\psi$ is a formula and $\vec x = x_1, \ldots, x_n$ are distinct individual variables, the \emph{abstract} $\lambda\vec x.\psi$ names the $n$-place predicate that $\psi$ expresses with $x_1, \ldots, x_n$ as its argument places. For example, $\lambda x.(P(x) \land Q(x,c))$ is `being $P$ and bearing $Q$ to $c$'. We also call an abstract a \emph{condition}.

The formula $\phi[X \mapsto \lambda\vec x.\psi]$ is $\phi$ with each atom $X(t_1,\ldots,t_n)$ replaced by $\psi[x_1 \mapsto t_1, \ldots, x_n \mapsto t_n]$, renaming bound variables to avoid capture. We write $\phi[X \mapsto R]$ for $\phi[X \mapsto \lambda \vec x.R(\vec x)]$. For $n = 0$, an abstract is just a sentence.

Bases change accordingly: rules may now contain predicate variables as well. These are instantiated by predicate constants of the same arity, and they may be constrained, in which case an eigenconstant is a predicate constant. A base mentions only finitely many individual and predicate constants, though it may mention infinitely many higher-type constants. Renamings act on individual and predicate constants, preserving arity, at all their occurrences, including occurrences as arguments of higher-type atoms, and they leave higher-type constants alone.

Higher-type constants are exempt from the finiteness requirement because none of them can be an eigenconstant. The variables of a rule stand for individual constants, predicate constants, and basic sentences, never for higher-type constants, so freshness concerns individual and predicate constants only. The exemption is needed in Section~\ref{sec:completeness-2}, where the simulation base uses infinitely many higher-type constants, one for each shape of formula.

With these readings, Definitions~\ref{def:atomic-rule}--\ref{def:derivation} and Lemmas~\ref{lem:renaming} and~\ref{lem:mono} carry over verbatim, and so do their proofs. The proofs choose fresh constants and compare occurrences of constants, and nothing in them depends on whether a constant is individual or predicate.

The obvious way to extend the substitutional reading to the second-order quantifier is to let $\foralltwo X\phi$ be supported when $\phi[X \mapsto R]$ is supported for every predicate constant $R$ of the right arity. This clause is too weak for van Dalen's logic. His rule for eliminating the second-order quantifier,
\[
\infer[\ern\foralltwo]{\phi[X \mapsto \lambda\vec x.\psi]}{\foralltwo X \phi}
\]
allows instantiation by an \emph{arbitrary} formula $\psi$~\cite[p.~147]{van1994logic}, which is full comprehension. Under the clause above, comprehension fails.

\begin{example}[Predicate constants]\label{ex:names-alone}
Take $\phi := X$ with $X$ of arity $0$, and $\psi := R(c)$. The rule then derives $R(c)$ from $\foralltwo X^0 X^0$. Now consider the base consisting of a single schematic axiom
\[
\infer{X}{}
\]
where $X$ is a $0$-ary predicate variable. Its instances derive every $0$-ary atom, so the base supports $\foralltwo X^0 X^0$ under the clause above, but it does not derive $R(c)$. So $\foralltwo X^0 X^0 \to R(c)$ is not valid. In particular, the formula $\foralltwo X^0 X^0$, which in van Dalen's logic is equivalent to $\bot$, does not have the meaning of absurdity.
\end{example}

What is missing is what Henkin's semantics requires of the range of the quantifier: closure under definability~\cite{henkin1950completeness}. The predicate constants of a language are not closed under definability in any sense.

Henkin's insight can be stated in a sentence: the second-order quantifier need not range over \emph{all} properties, only over a designated stock of them, provided the stock is closed under definability. Nothing more is required for the logic, and nothing less will do. The account given here has the same shape, with inference in place of reference, and it has two parts. We illustrate both with the vixen base $\base{V}$ of Example~\ref{ex:vixen-support}, which serves as a running example from here to the end of Section~\ref{sec:completeness-2}.

\emph{First, `for every property' means `for every concept the base has'.} The concepts of a base are those its rules can be about. Since rules are about atoms, these are the \emph{atomic predicates}: whatever can be said of some objects by a single atomic sentence with gaps in it. In the vixen base, `is a vixen', `is female', and `is a fox' are atomic predicates. The meaning of each is exhausted by its inferential role in the base. The second-order quantifier ranges over these: $\foralltwo X\phi$ is supported just when $\phi$ holds of each atomic predicate.

A compound condition is not an atomic predicate. `Is a female fox', that is, $\lambda x.(F_e(x) \land F_o(x))$, is expressed by a formula with a connective in it, and no rule of a base has such a formula as premiss or conclusion, so the quantifier does not range over it.

The clause is substitutional, as the first-order clause is, and it is \emph{predicative} in the following sense. The instances are less complex than $\foralltwo X \phi$, so the meaning of $\foralltwo X\phi$ is explained in terms of strictly simpler judgements, and the explanation does not presuppose an understanding of the totality quantified over. The inferential role of an atomic predicate may of course be tied by the base to complex conditions, as in the vixen example. That is where the impredicativity will reside.

\emph{Second, for `every property' to cover every condition the language can express, the base must have a concept for each such condition.} Any condition $\psi$ can be given one, by a predicate $R$ governed by two rules: from $\psi(\vec a)$ infer $R(\vec a)$, and from $R(\vec a)$ infer $\psi(\vec a)$. This is the paradigm of a definition in Dummett's sense. The introduction and elimination rules are in perfect harmony, each undoing the other, and the new predicate means precisely what $\psi$ means~\cite[Ch.~9]{Dummett1991logical}.

The vixen is an instance: the rules of $\base{V}$ make `is a vixen' a predicate that means precisely `is a female fox'. We call such a stand-in a \emph{proxy}. We call a base \emph{comprehensive} when every condition expressible in the language has a proxy among its atomic predicates, and robustly so: the stand-in and the condition agree not only in the base as it is but in every extension of it.

Comprehensiveness is the inferentialist analogue of Henkin's closure under definability. Henkin asks that the stock of properties contain every definable property. We ask that the stock of concepts of a base contain a definitional proxy for every expressible condition.

Together, the two parts give comprehension. If $\foralltwo X\phi$ is supported, then $\phi$ holds of every concept, in particular of the proxy for $\psi$. Since the proxy and $\psi$ cannot be told apart in any extension of the base, $\phi$ holds of $\psi$ too (Lemma~\ref{lem:replacement}).

Full impredicative comprehension is thus validated by a predicative meaning-explanation. The impredicative strength lies not in what the quantifier means but in which bases count: those whose atomic predicates are closed under definition. In Henkin's semantics likewise, the clause for the quantifier is the plain `for every member of the domain', and comprehension is a condition on which structures count as models.

Every base extends to a comprehensive one by adding definitions (Corollary~\ref{cor:comprehensive}). Whether they can be added without altering what the base already supports in its old vocabulary is a further question, to which Section~\ref{sec:discussion} returns.

We now make the first part precise.

\begin{definition}[Atomic predicate]\label{def:atomic-predicate}
For $n \geq 0$, an \emph{atomic predicate} of arity $n$ is an abstract $\lambda \vec x.A$, where $\vec x = x_1, \ldots, x_n$ and $A$ is an atom, ordinary or higher-type, whose free variables are among $\vec x$. We write $\Ain_n$ for the set of atomic predicates of arity $n$.
\end{definition}

The letter $\Ain$ is for `atomic'. Informally, $\Ain_n$ comprises what can be said of $n$ things by a single atomic sentence with gaps in it. Besides the predicate constants themselves, it contains the predicates obtained from a single atom by filling, identifying, permuting, or adding places, such as $\lambda x.R(x,c)$, $\lambda x.R(x,x)$, and $\lambda x y.R(y,x)$. It also contains higher-type atoms with gaps, such as $\lambda x.\mathsf{Ex}(\mathsf{Br}, x)$, `exemplifying bravery'.

What $\Ain_n$ does \emph{not} contain are compound predicates such as $\lambda x.(P(x) \land Q(x))$. Whether a base has an atomic stand-in for such a predicate is what comprehensiveness is about. For $n = 0$ the atomic predicates are the basic sentences. Instantiating $X$ by an atomic predicate replaces each atom $X(\vec t)$ by an atom, so it does not change the degree.

\begin{definition}[Support for second-order formulae]\label{def:support-2}
Support for closed formulae of the second-order language is given by the clauses of Figure~\ref{fig:support}, with $p$ ranging over all basic sentences, together with
\[
\supp_{\base{B}} \foralltwo X \phi \quad\text{iff}\quad \supp_{\base{B}} \phi[X \mapsto \lambda\vec x.A] \text{ for every } \lambda\vec x.A \in \Ain_n, \tag{$\foralltwo$}
\]
where $n$ is the arity of $X$.
\end{definition}

The recursion is still well founded, since instances by atomic predicates have smaller degree than $\foralltwo X \phi$. With this clause, $\foralltwo X^0 X^0$ is supported in $\base{B}$ iff every basic sentence is derivable in $\base{B}$. It has exactly the support condition of $\bot$, which is Sandqvist's. The base $\{\infer{X}{}\}$ of Example~\ref{ex:names-alone} no longer supports it.

\begin{example}[A second-order formula]\label{ex:eta}
In the language of the vixen base, the formula
\[
\eta := \foralltwo X\big(\forall x(F_e(x) \land F_o(x) \to X(x)) \to X(t)\big)
\]
says that whatever holds of every female fox holds of Tammy, and the base $\base{V}$ supports it. By $(\foralltwo)$ we must show, for each atomic predicate $\lambda x.A$, that $\base{V}$ supports
\[
\forall x(F_e(x) \land F_o(x) \to A) \to A[x \mapsto t] .
\]
Let $\base{C}$ extend $\base{V}$ and support the antecedent, and so its instance $F_e(t) \land F_o(t) \to A[x \mapsto t]$. Since $\base{C}$ derives $V(t)$, it derives $F_e(t)$ and $F_o(t)$, and so supports $F_e(t) \land F_o(t)$. Hence it supports $A[x \mapsto t]$.

The clause $(\foralltwo)$ thus yields the instance of $\eta$ at `is a vixen', at `is a fox', and at every other atomic predicate. It does not by itself yield the instance at the compound condition `is a female fox', which is a step of comprehension and needs a proxy.
\end{example}

The basic properties of Section~\ref{sec:support} carry over, each with one new case for $\foralltwo X\phi$.
\begin{itemize}[label=--]
  \item \emph{Monotonicity} (Lemma~\ref{lem:mono-supp}). Apply the induction hypothesis to the instances by atomic predicates, as for $\forall x\phi$.
  \item \emph{Equivariance} (Lemma~\ref{lem:equiv-supp}). Apply the induction hypothesis to the instances, and note that a bijective renaming permutes each $\Ain_n$.
  \item \emph{Conservativity} (Lemma~\ref{lem:conservative}). The lemma concerns only derivability, and carries over unchanged.
  \item \emph{Elimination} (Lemma~\ref{lem:elim}). The induction gains the case $\chi = \foralltwo Y \chi'$, which is word for word the case $\chi = \forall y\chi'$ with instances by atomic predicates in place of instances by constants.
\end{itemize}

We now make the second part precise. A proxy for a condition is an atomic predicate that agrees with it in every extension of the base, as the running example shows.

\begin{example}[A proxy]\label{ex:vixen-proxy}
In the vixen base $\base{V}$ of Example~\ref{ex:vixen-support}, the atomic predicate $\lambda x.V(x)$ is a proxy for $\lambda x.(F_e(x) \land F_o(x))$. Indeed, let $\base{C}$ be an extension of $\base{V}$ and $a$ a constant. If $\base{C}$ derives $V(a)$, it derives $F_e(a)$ and $F_o(a)$ by the first two general rules of $\base{V}$. Conversely, if it derives both, it derives $V(a)$ by the third. By (At) and $(\land)$, $\supp_{\base{C}} V(a)$ iff $\supp_{\base{C}} F_e(a) \land F_o(a)$. The rules for `vixen' are a definition of `female fox', in the sense explained above.
\end{example}

Agreement must hold in every extension, and not merely in the base itself. The reason is that the clauses for implication, disjunction, and the existential quantifier consult extensions. Two predicates that agree in the base but come apart in an extension are distinguished by formulae built with those clauses, and then one cannot replace the other.

\begin{example}[Coincidence]\label{ex:coincidence}
Let $\base{V}'$ consist of three axioms and no other rules:
\[
\base{V}' := \left\{\; \infer{V(t)}{} \;,\quad \infer{F_e(t)}{} \;,\quad \infer{F_o(t)}{} \;\right\}
\]
Tammy is the only known vixen, and the only individual known to be female and a fox. In $\base{V}'$ itself, `is a vixen' and `is a female fox' agree: $\supp_{\base{V}'} V(a)$ iff $a$ is $t$ iff $\supp_{\base{V}'} F_e(a) \land F_o(a)$. But the agreement is a coincidence of what happens to be known. For another name $b$, the extension $\base{V}' \cup \{\infer{F_e(b)}{}, \infer{F_o(b)}{}\}$ supports $F_e(b) \land F_o(b)$ and does not derive $V(b)$.

The difference shows in implications: the base $\base{V}'$ supports $V(b) \to V(b)$, as every base does, but it does not support $(F_e(b) \land F_o(b)) \to V(b)$, since the extension just described supports the antecedent and not the consequent. So replacing `is a vixen' by `is a female fox' in a formula supported by $\base{V}'$ may yield one that is not. In $\base{V}$ this cannot happen (Lemma~\ref{lem:replacement}).
\end{example}

The examples separate three things, which the formal definition must keep apart.
\begin{enumerate}[label=(\arabic*)]
  \item \emph{A condition is expressible.} This is a fact about the language. The condition $\lambda x.(F_e(x) \land F_o(x))$ is expressible whatever the base, even in the empty one.
  \item \emph{An atomic predicate agrees with the condition in the base.} This is a fact about what the base presently derives. It holds of `is a vixen' in both $\base{V}$ and $\base{V}'$.
  \item \emph{The atomic predicate and the condition are inferentially equivalent in every extension.} This is a fact about the rules that govern the predicate. It holds in $\base{V}$ and fails in $\base{V}'$.
\end{enumerate}
Only the third makes an atomic predicate a proxy.

\begin{definition}[Proxy; comprehensive base]\label{def:proxy}
Let $\psi$ be a formula whose free variables are among the individual variables $\vec x = x_1, \ldots, x_n$. An atomic predicate $\lambda\vec x.A \in \Ain_n$ is a \emph{proxy} for $\lambda\vec x.\psi$ in the base $\base{B}$ if
\[
\supp_{\base{C}} A[\vec x \mapsto \vec a] \quad\text{iff}\quad \supp_{\base{C}} \psi[\vec x \mapsto \vec a]
\]
for every base $\base{C}$ extending $\base{B}$ and all constants $\vec a = a_1, \ldots, a_n$. A base is \emph{comprehensive} if every such $\lambda\vec x.\psi$ has a proxy in it.
\end{definition}

The vixen base supplies a proxy for one condition. A comprehensive base supplies one for every condition: for `is female or a fox', for `is not a fox', and for every other condition that the language can express, including those that themselves contain second-order quantifiers.

Since a proxy is required to work in every extension, every extension of a comprehensive base is comprehensive. Hence the clauses that quantify over extensions only ever visit comprehensive bases once they start from one, and it makes sense to restrict validity to them.

Comprehensiveness corresponds to Henkin's closure condition point by point. In a Henkin structure, closure under definability requires that the relation defined by each formula belong to the range of the quantifier. Comprehensiveness requires that the condition expressed by each formula have a proxy in $\Ain_n$. `Is a proxy for' takes the place of `is the extension of', and it is the stronger tie: a proxy answers to its condition in every base extending the given one, whereas the extension of a formula answers to the formula in a single structure.

\begin{definition}[H-validity]\label{def:H-validity}
Let $\Gamma$ be a finite set of closed formulae and $\phi$ a closed formula. Then $\phi$ is an \emph{H-consequence} of $\Gamma$, written $\Gamma \supp_{\mathrm{H}} \phi$ (H for Henkin), if $\Gamma \supp_{\base{B}} \phi$ for every comprehensive base $\base{B}$.
\end{definition}

A second-order structure is a model when the comprehension schema holds in it~\cite[Def.~5.3]{van1994logic}, and validity is truth in all models. The system van Dalen gives is complete for this notion of validity~\cite[Thm.~5.4]{van1994logic}. Comprehensive bases exist, and indeed every base has a comprehensive extension (Corollary~\ref{cor:comprehensive}).

The restriction to comprehensive bases cannot be dropped, because over arbitrary bases the rule $\ern\foralltwo$ is not sound.

\begin{example}[No proxy]\label{ex:not-sound}
Let $Q$ and $R$ be $0$-ary predicate constants, and let $\base{B}$ consist of the single rule
\[
\infer{Q}{Z}
\]
where $Z$ is an atom variable. The rule says that $Q$ follows from any basic sentence whatever.

The base $\base{B}$ supports $\foralltwo X^0(X^0 \to Q)$. By $(\foralltwo)$ and $(\to)$ we must show, for each basic sentence $p$, that every extension of $\base{B}$ deriving $p$ derives $Q$, and this holds by the instance of the rule with $Z := p$.

The base $\base{B}$ does not support $(R \to R) \to Q$. Every base supports $R \to R$, so $\base{B}$ supports $(R \to R) \to Q$ only if it derives $Q$. But $\base{B}$ derives no basic sentence at all, because every derivation legal in $\base{B}$ has an open assumption: a one-node derivation has one, and the rule of $\base{B}$ discharges nothing.

Yet $\ern\foralltwo$, instantiating $X^0$ by the formula $R \to R$, derives $(R \to R) \to Q$ from $\foralltwo X^0(X^0 \to Q)$. What $\base{B}$ lacks is a proxy for the condition $R \to R$. Such a proxy would be a basic sentence derivable in every extension of $\base{B}$, and so in $\base{B}$ itself.
\end{example}

\begin{remark}[Overflow]\label{rem:overflow}
The clause $(\foralltwo)$ validates instantiation by atomic predicates in every base. Example~\ref{ex:not-sound} shows that it does not validate instantiation by arbitrary conditions, and so that the overflow argument of Lemma~\ref{lem:elim} has no counterpart for the second-order quantifier.

The comparison with Lemma~\ref{lem:elim} can be made precise. The clause $(\lor)$ is the support condition, in every base, of the formula $\foralltwo Y\big((\phi \to Y) \to ((\psi \to Y) \to Y)\big)$, where $Y$ is a $0$-ary predicate variable not free in $\phi$ or $\psi$. The clauses $(\bot)$ and $(\exists)$ are likewise the support conditions of the usual second-order definitions. So Sandqvist's clauses give these constants by their atomic instances, and Lemma~\ref{lem:elim} shows that the atomic instances suffice: the passage from basic sentences to arbitrary formulae goes through by induction on the formula.

For an arbitrary $\foralltwo X\phi$ the corresponding induction would run on the condition put for $X$, and it fails. The induction of Lemma~\ref{lem:elim} depends on the particular form of the three definitions, and nothing comparable holds in general. In Example~\ref{ex:not-sound} the variable $X$ stands as the antecedent of an implication, and nothing about $(R \to R) \to Q$ follows from the instances $p \to Q$. Proxies are what the semantics uses in place of the missing induction.
\end{remark}

It remains to fix the calculus. Let $\NItwo$ be the natural deduction system of Figure~\ref{fig:NI} together with the two rules
\[
\infer[\irn{\foralltwo}]{\foralltwo X \phi}{\phi[X \mapsto E]}
\qquad\qquad
\infer[\ern{\foralltwo}]{\phi[X \mapsto \lambda \vec x.\psi]}{\foralltwo X \phi}
\]
where $E$ is a predicate constant of the arity of $X$ occurring neither in $\foralltwo X\phi$ nor in any open assumption on which the premiss rests, and $\psi$ is any formula whose free variables are among $\vec x$. This is van Dalen's system~\cite[Ch.~5]{van1994logic} without the rule of classical reductio, with predicate constants as eigenvariables. We write $\Gamma \Ntwodash \phi$ as before.

In $\NItwo$ the constants $\bot$, $\land$, $\lor$, and $\exists$ are definable from $\to$, $\forall$, and $\foralltwo$, as Russell and Prawitz observed~\cite{Prawitz2006natural}. For instance, $\phi \lor \psi$ is interderivable with $\foralltwo Y\big((\phi \to Y) \to ((\psi \to Y) \to Y)\big)$, for $Y$ a $0$-ary predicate variable not free in $\phi$ or $\psi$~\cite[Thm.~5.5]{van1994logic}. Taking the connectives as primitive therefore loses nothing.

The second-order existential quantifier is defined by
\[
\existstwo X \phi \;:=\; \foralltwo Y\big(\foralltwo X(\phi \to Y) \to Y\big).
\]
Writing $\phi \leftrightarrow \psi$ for $(\phi \to \psi) \land (\psi \to \phi)$, $\NItwo$ proves the comprehension schema $\existstwo X \forall \vec x(X(\vec x) \leftrightarrow \psi)$ for every formula $\psi$ in which $X$ is not free~\cite[Ch.~5]{van1994logic}.

\section{Second-order Soundness}\label{sec:soundness-2}

The new case is comprehension. Suppose $\foralltwo X\phi$ is supported in a comprehensive base, and let $\lambda\vec x.\psi$ be a condition. The clause $(\foralltwo)$ gives $\phi$ at every atomic predicate, in particular at a proxy for the condition. What we need is $\phi$ at the condition itself. This section first shows that a proxy can stand in for its condition anywhere in a formula, and then proves soundness.

The replacement lemma is the formal version of the remark that a defined predicate means precisely what its definition says. In the vixen base, take the body of the formula $\eta$ of Example~\ref{ex:eta}. Its instance at the proxy `is a vixen' and its instance at the condition `is a female fox' are
\[
\begin{gathered}
\forall x(F_e(x) \land F_o(x) \to V(x)) \to V(t) \\
\forall x(F_e(x) \land F_o(x) \to F_e(x) \land F_o(x)) \to F_e(t) \land F_o(t)
\end{gathered}
\]
The lemma says that the two are supported by exactly the same extensions of $\base{V}$, and that the same holds for every formula in place of the body of $\eta$.

\begin{lemma}[Replacement]\label{lem:replacement}
Let $\lambda\vec x.A$ be a proxy for $\lambda\vec x.\psi$ in $\base{B}$, and let $\phi$ be a formula whose only free variable is the $n$-ary predicate variable $X$. Then, for every base $\base{C}$ extending $\base{B}$,
\[
\supp_{\base{C}} \phi[X \mapsto \lambda\vec x.A] \quad\text{iff}\quad \supp_{\base{C}} \phi[X \mapsto \lambda\vec x.\psi].
\]
\end{lemma}

The proof is an induction on $\phi$. It is routine once the statement quantifies over all extensions of $\base{B}$: each clause of support refers only to the base at hand and to its extensions, and the proxy agrees with the condition at all of them. Example~\ref{ex:coincidence} shows that the induction would fail at the clause for implication if agreement were required only at $\base{B}$.

\begin{proof}
By induction on the degree of $\phi$, for all extensions $\base{C}$ of $\base{B}$ at once.
\begin{itemize}[label=--]
  \item $\phi = X(\vec a)$. The claim is the defining property of a proxy.
  \item $\phi$ \emph{another atom, or} $\bot$. The substitution leaves $\phi$ unchanged, since $X$ cannot occur as the argument of a higher-type atom.
  \item $\phi = \phi_1 \land \phi_2$. By the induction hypothesis at $\base{C}$.
  \item $\phi = \phi_1 \to \phi_2$ \emph{or} $\phi = \phi_1 \lor \phi_2$. The clause refers to $\phi_1$ and $\phi_2$ at extensions of $\base{C}$. These extend $\base{B}$, so the induction hypothesis applies at each of them.
  \item $\phi = \forall y\phi'$ \emph{or} $\phi = \exists y\phi'$. Instantiating $y$ by a constant commutes with the substitution for $X$, and yields a formula of smaller degree whose only free variable is $X$. The induction hypothesis applies to each instance, at $\base{C}$ or at its extensions.
  \item $\phi = \foralltwo Y\phi'$. Instantiating $Y$ by an atomic predicate commutes with the substitution for $X$, since atomic predicates do not contain $X$. It yields a formula of smaller degree whose only free variable is $X$, and the induction hypothesis applies. \qedhere
\end{itemize}
\end{proof}

For the rest of the soundness proof we proceed as in Section~\ref{sec:soundness}. The rule $\irn\foralltwo$ raises the same difficulty as $\irn\forall$: its premiss is established for an eigenconstant $E$, but the clause $(\foralltwo)$ needs every atomic predicate. So we again prove a stronger statement, this time stable under maps that replace predicate constants by atomic predicates.

An \emph{instantiation} $\theta$ sends each individual constant to an individual constant and each predicate constant of arity $n$ to an atomic predicate in $\Ain_n$. It acts on a formula by replacing every occurrence of each individual constant $a$ by $\theta(a)$, and every atom $R(\vec t)$ by $\theta(R)$ applied to the arguments. Predicate constants that occur as arguments of higher-type atoms are left unchanged. A renaming that fixes every predicate constant is an instantiation.

\begin{definition}[Stable H-validity]\label{def:stable-H}
The pair $(\Delta : \phi)$ is \emph{stably H-valid} if $\Delta\theta \supp_{\base{B}} \phi\theta$ for every instantiation $\theta$ and every comprehensive base $\base{B}$; that is, if it is stably valid in the sense of Definition~\ref{def:stable}, with instantiations in place of renamings and comprehensive bases in place of arbitrary ones.
\end{definition}

With these notions the theorem can be stated and proved.

\begin{theorem}[Soundness]\label{thm:soundness-H}
If $\Gamma \Ntwodash \phi$, then $\Gamma \supp_{\mathrm{H}} \phi$.
\end{theorem}
\begin{proof}
As for Theorem~\ref{thm:soundness}, we show by induction on an $\NItwo$-derivation that the pair consisting of its open assumptions and its conclusion is stably H-valid. The theorem follows by taking $\theta$ to be the identity. Instantiations commute with the connectives and quantifiers, and $(\phi[x \mapsto a])\theta = (\phi\theta)[x \mapsto \theta a]$.

The cases for the rules of Figure~\ref{fig:NI} are those of the proof of Theorem~\ref{thm:soundness}, with $\theta$ in place of $\tau$ and comprehensive bases in place of arbitrary ones. That proof used two facts about the class of bases, and both hold here. First, the clauses $(\to)$, $(\lor)$, $(\exists)$, and (Inf) must visit only bases of the class, and extensions of a comprehensive base are comprehensive. Second, Lemma~\ref{lem:elim} must hold, and it holds for the extended language (Section~\ref{sec:second-order}). In the cases of $\irn\forall$ and $\ern\exists$, the instantiation $\theta'$ agrees with $\theta$ except that it sends the eigenconstant $e$ to the given constant $a$.

It remains to treat the two new rules. As before, let $\Gamma$ be the open assumptions, fix $\theta$, and fix a comprehensive base $\base{C}$ supporting every member of $\Gamma\theta$.
\begin{itemize}[label=--]
  \item $\irn\foralltwo$. Let the premiss be $\phi[X \mapsto E]$, resting on assumptions $\Gamma$, with $E$ in neither $\Gamma$ nor $\foralltwo X\phi$. Let $\lambda\vec x.A$ be an atomic predicate of the arity of $X$, and let $\theta'$ agree with $\theta$ except that $\theta'(E) = \lambda\vec x.A$. The constant $E$ occurs in $\phi[X \mapsto E]$ only where $X$ occurred, and these are predicate positions, not arguments of higher-type atoms. It does not occur in $\Gamma$. Hence $\Gamma\theta' = \Gamma\theta$ and $(\phi[X \mapsto E])\theta' = (\phi\theta)[X \mapsto \lambda\vec x.A]$. Stable H-validity of the premiss, applied to $\theta'$, shows that $\base{C}$ supports $(\phi\theta)[X \mapsto \lambda\vec x.A]$. As the atomic predicate was arbitrary, $\base{C}$ supports $\foralltwo X(\phi\theta) = (\foralltwo X\phi)\theta$.
  \item $\ern\foralltwo$. The premiss gives $\supp_{\base{C}} \foralltwo X(\phi\theta)$, and the conclusion is $\phi[X \mapsto \lambda\vec x.\psi]$. So we must show that $\base{C}$ supports $(\phi\theta)[X \mapsto \lambda\vec x.\psi\theta]$. Since $\base{C}$ is comprehensive, $\lambda\vec x.\psi\theta$ has a proxy $\lambda\vec x.A$ in $\base{C}$. By $(\foralltwo)$, $\base{C}$ supports $(\phi\theta)[X \mapsto \lambda\vec x.A]$. By Lemma~\ref{lem:replacement}, taking the extension to be $\base{C}$ itself, it supports $(\phi\theta)[X \mapsto \lambda\vec x.\psi\theta]$. \qedhere
\end{itemize}
\end{proof}

The case $\ern\foralltwo$ is the precise form of the informal argument of Section~\ref{sec:second-order}. The quantifier gives every concept, comprehensiveness supplies a concept for $\psi$, and replacement transfers the result from the concept to $\psi$. It is the only case that uses comprehensiveness.

\section{Second-order Completeness}\label{sec:completeness-2}

This section proves the second main result.

\begin{theorem}[Completeness]\label{thm:B}
Let $\Gamma$ be a finite set of closed second-order formulae and $\zeta$ a closed second-order formula. Then $\Gamma \supp_{\mathrm{H}} \zeta$ iff $\Gamma \Ntwodash \zeta$.
\end{theorem}

Soundness is Theorem~\ref{thm:soundness-H}. For completeness we follow the plan of Section~\ref{sec:completeness}: counterparts, a simulation base, simulation, translation. Throughout this section we fix a finite set $\Gamma$ of closed formulae and a closed formula $\zeta$.

The construction must produce one base, $\base{N}_{\mathrm{H}}$, with three properties.
\begin{enumerate}[label=(\arabic*)]
  \item \emph{Simulation.} In $\base{N}_{\mathrm{H}}$ and its extensions, a formula is supported exactly when its counterpart is derivable.
  \item \emph{Translation.} Every derivation in $\base{N}_{\mathrm{H}}$ becomes an $\NItwo$-derivation when counterparts are replaced by the formulae they stand for.
  \item \emph{Comprehensiveness.} The base $\base{N}_{\mathrm{H}}$ is comprehensive.
\end{enumerate}
The first two are the requirements of Section~\ref{sec:completeness}. The third is new, and it is forced by the definition of H-validity. From $\Gamma \supp_{\mathrm{H}} \zeta$ we learn something about $\base{N}_{\mathrm{H}}$ only if $\base{N}_{\mathrm{H}}$ is among the comprehensive bases.

The third requirement changes everything else, because a comprehensive base contains a proxy for every condition expressible in the language. The simulation base can no longer confine itself to the formulae relevant to $\Gamma$ and $\zeta$. It must simulate the \emph{whole} language, and the counterparts of formulae will themselves be the proxies. In the terms of Section~\ref{sec:second-order}, the simulation base is a base in which a concept has been coined for every condition. This is the inferentialist reading of van Dalen's proof of completeness for Henkin models, which adds predicates $Ap_n$ and a comprehension axiom guaranteeing every formula a representative~\cite[pp.~148--149]{van1994logic}.

The construction of Section~\ref{sec:completeness} does not meet the third requirement as it stands. It fails at four points, and each failure calls for one of the devices introduced below.
\begin{enumerate}[label=(\alph*)]
  \item \emph{Predicate constants.} A second-order formula may contain predicate constants, and its counterpart must display them, so that the eigenvariable condition of $\irn\foralltwo$ can be read off the counterpart. An ordinary atom cannot take a predicate constant as an argument, so counterparts are \emph{higher-type atoms}.
  \item \emph{All shapes.} The base must contain rules for every shape of formula, not for finitely many, and is therefore infinite. It remains a base because its rules mention no individual or predicate constants.
  \item \emph{Counterparts inside formulae.} A proxy is an atomic predicate of the language, so counterparts are atoms of the language and occur in formulae, which need counterparts in turn. Nothing so far prevents the counterpart of a formula from occurring in that very formula, and then replacing counterparts by the formulae they stand for would never end. We prevent this by arranging the heads of counterparts in \emph{levels}.
  \item \emph{No inverse.} An atom is its own counterpart. A counterpart is an atom, so it is the counterpart both of itself and of the formula it was introduced for. The counterpart map is therefore not injective and has no inverse. The translation uses instead a substitution, the \emph{unwinding}, which replaces counterparts by formulae from the highest level downwards.
\end{enumerate}

We take the devices in turn. As before, the counterpart of a formula is an atom that carries the formula's constants as arguments. It must be an atom because only atoms occur in the rules of a base, and because the range of the second-order quantifier consists of atomic predicates. If every closed formula has an atomic counterpart, uniformly in its constants, then every condition has an atomic predicate that tracks it: the counterpart of its instances, with gaps where the arguments go.

Shapes now abstract predicate constants as well as individual ones. The \emph{shape} $\hat\gamma$ of a closed formula $\gamma$ is obtained by replacing each occurrence of an individual or predicate constant, including its occurrences as arguments of higher-type atoms, by a distinct new variable of the same sort, in left-to-right order. As before, $\vec s_\gamma$ lists the constants replaced, occurrence by occurrence. Shapes contain no individual or predicate constants, though they may contain higher-type constants. For example, the shape of $\forall x(R(x,c) \to P(b))$ is $\forall x(Y_1(x,u_1) \to Y_2(u_2))$, with list of constants $R, c, P, b$.

The counterpart of $\gamma$ will be a higher-type atom $h(\vec s_\gamma)$, whose \emph{head} $h$ depends only on the shape of $\gamma$. The heads must be chosen with care, because of point~(c) above. The formula $h(a,a) \land P(b)$, where $h(a,a)$ is the counterpart of some formula, needs a counterpart of its own, and unwinding it must not lead back to $h$.

We therefore arrange the heads in levels. Partition each $\HCONS_\kappa$ into infinitely many infinite \emph{levels} $\HCONS_\kappa^{(0)}, \HCONS_\kappa^{(1)}, \ldots$. Assign to each shape $\hat\gamma$ of a nonatomic closed formula a \emph{head} $h_{\hat\gamma}$, injectively, such that:
\begin{itemize}[label=--]
  \item $h_{\hat\gamma}$ has the signature of $\vec s_\gamma$, that is, it takes arguments of the sorts of the constants abstracted;
  \item $h_{\hat\gamma}$ does not occur in $\Gamma \cup \{\zeta\}$;
  \item the level of $h_{\hat\gamma}$ exceeds the level of every head occurring in $\hat\gamma$.
\end{itemize}
The third condition is the one that answers point~(c): a formula never contains the head of its own counterpart, nor any head whose formula contains it, and so on.

Such an assignment exists: enumerate the countably many shapes and assign heads one at a time; at each stage only finitely many heads of the required level and signature are already used or occur in $\Gamma \cup \{\zeta\}$, and each level is infinite. Heads not assigned to any shape are \emph{primitive}.

\begin{definition}[Counterpart]\label{def:counterpart-2}
For a closed formula $\gamma$, its \emph{counterpart} is $\flatop{\gamma} := \gamma$ if $\gamma$ is atomic, and $\flatop{\gamma} := h_{\hat\gamma}(\vec s_\gamma)$ otherwise.
\end{definition}

In Section~\ref{sec:completeness} the translation back to natural deduction used the inverse of the counterpart map. Here there is no inverse, and its place is taken by a substitution.

\begin{definition}[Unwinding]\label{def:unwinding}
The \emph{unwinding} $\natop{(-)}$ is the substitution, defined on all formulae by recursion on the highest level of assigned heads occurring, that replaces each atom $h_{\hat\gamma}(\vec s)$ with an assigned head by the unwinding of $\hat\gamma[\vec u \mapsto \vec s]$, and leaves everything else unchanged.
\end{definition}

Here $\hat\gamma[\vec u \mapsto \vec s]$ is the formula obtained by putting the arguments $\vec s$ back into the shape, predicate constants into predicate positions and individual terms into individual positions. All the assigned heads it contains have lower level than $h_{\hat\gamma}$, so the recursion terminates. Shapes themselves only index heads; they are never unwound.

The unwinding recovers the original formula in the following sense. If $\gamma$ contains no assigned heads, then $\natop{(\flatop{\gamma})} = \gamma$: one step of unwinding undoes the counterpart map. If $\gamma$ does contain assigned heads, the unwinding goes further and replaces those as well, so that $\natop{(\flatop{\gamma})} = \natop{\gamma}$. In either case the result contains no assigned heads. The formulae of $\Gamma \cup \{\zeta\}$ are of the first kind, and that is all the completeness proof needs.

\begin{example}[Levels]\label{ex:layers}
Let $\gamma := \foralltwo X(X(a) \to X(a))$, whose shape is $\foralltwo X(X(u_1) \to X(u_2))$. Then $\flatop{\gamma} = h(a,a)$ for a head $h$ taking two individual arguments, at level $0$ say. Now $h(a,a)$ is itself a basic sentence, so $\gamma' := h(a,a) \land P(b)$ is a closed formula. Its shape is $h(u_1,u_2) \land Y(u_3)$, which contains $h$, so its head $g$ has level at least $1$, and $\flatop{\gamma'} = g(a,a,P,b)$.

Unwinding $g(a,a,P,b)$ gives $\natop{(h(a,a))} \land P(b)$, and unwinding $h(a,a)$ gives $\gamma$. So $\gamma'$ and its counterpart unwind to the same formula, in which no assigned head occurs:
\[
\begin{array}{ccc}
h(a,a) \land P(b) & \stackrel{\flat}{\longmapsto} & g(a,a,P,b)\\[2mm]
{\scriptstyle\sharp}\searrow & & \swarrow{\scriptstyle\sharp}\\[2mm]
\multicolumn{3}{c}{\foralltwo X(X(a) \to X(a)) \land P(b)}
\end{array}
\]
\end{example}

The example shows that $\flatop{(-)}$ is not injective: $h(a,a)$ is its own counterpart and also the counterpart of $\gamma$. Injectivity served two purposes in Section~\ref{sec:completeness}, and neither needs it here. It gave the translation an inverse to work with, and the unwinding now does that work. It also ensured that the translated rules discharge exactly the right assumptions. Here they may discharge more, which only removes open assumptions (Lemma~\ref{lem:ddagger-H}).

The failure of injectivity is also what makes counterparts usable as proxies, because the atom $h(a,a)$ is at once a sentence of the language, within the range of the quantifier, and the stand-in for $\gamma$.

The next lemma collects the properties of counterparts and their unwinding on which the rest of the section depends. Parts~(a) and~(b) say that neither map adds or removes constants, and that unwinding is compatible with the substitutions that the rules perform. They are what the translation of eigenvariable conditions requires. Part~(c) says that the translation ends where it should. Part~(d) produces, for every condition, the atomic predicate that will be its proxy.

\begin{lemma}[Counterparts]\label{lem:flat-H}
For every closed formula $\gamma$:
\begin{enumerate}[label=(\alph*)]
  \item $\con{\flatop{\gamma}} = \con{\gamma}$;
  \item $\natop{(\flatop{\gamma})} = \natop{\gamma}$; the unwinding commutes with instantiation by constants and by atomic predicates, and preserves $\con{-}$;
  \item $\natop{\gamma} = \gamma$ if $\gamma \in \Gamma \cup \{\zeta\}$;
  \item for every formula $\psi$ whose free variables are among $\vec x = x_1, \ldots, x_n$, there is an atomic predicate $\lambda\vec x.A \in \Ain_n$ such that $A[\vec x \mapsto \vec a] = \flatop{(\psi[\vec x \mapsto \vec a])}$ for all constants $\vec a$.
\end{enumerate}
\end{lemma}
\begin{proof}
We take the four parts in turn.

(a) The list $\vec s_\gamma$ consists of exactly the constants of $\gamma$.

(b) The first claim is the definition of the unwinding on $h_{\hat\gamma}(\vec s_\gamma)$, since $\hat\gamma[\vec u \mapsto \vec s_\gamma] = \gamma$. For the rest, note that the unwinding rewrites only atoms with assigned heads. The formula it substitutes for $h_{\hat\gamma}(\vec s)$ has exactly the constants among $\vec s$, by induction on the level, since $\hat\gamma$ contains no individual or predicate constants. So the unwinding preserves $\con{-}$ and commutes with instantiation by constants. Commutation with instantiation by an atomic predicate reads $\natop{(\phi[X \mapsto \lambda\vec x.A])} = (\natop{\phi})[X \mapsto \lambda\vec x.\natop{A}]$, and holds because $X$ never occurs as an argument of a higher-type atom.

(c) No head occurring in $\Gamma \cup \{\zeta\}$ is assigned.

(d) If $\psi$ is atomic, take $A := \psi$. Otherwise the shape of $\psi[\vec x \mapsto \vec a]$ does not depend on $\vec a$, because shapes abstract occurrences and not constants. Take $A := h(\vec s)$, where $h$ is the head assigned to the shape of $\psi[\vec x \mapsto \vec a]$ and $\vec s$ is the list of constants of $\psi[\vec x \mapsto \vec a]$ with each occurrence of $a_i$ that came from $x_i$ replaced by $x_i$.
\end{proof}

Part (d) is where comprehensiveness will come from. Every condition $\lambda\vec x.\psi$ has an atomic predicate whose instances are the counterparts of its instances, namely its counterpart with gaps. Corollary~\ref{cor:comprehensive} shows that this atomic predicate is a proxy.

We can now define the simulation base. The base $\base{N}_{\mathrm{H}}$ consists of the rules of Figure~\ref{fig:N} for \emph{every} shape of a nonatomic closed formula, not only for those of $\Gamma$ and $\zeta$, with counterparts now as in Definition~\ref{def:counterpart-2} and the atom variable $Z$ ranging over all basic sentences. In addition, for every shape $\foralltwo X\phi$ it contains the rules
\[
\infer[\irn\foralltwo]{\flatop{(\foralltwo X\phi)}}{\flatop{(\phi[X \mapsto Y])}}\ \fresh{Y}
\qquad\qquad
\infer[\ern\foralltwo]{\flatop{(\phi[X \mapsto \lambda\vec x.\hat A])}}{\flatop{(\foralltwo X\phi)}}
\]

In $\irn\foralltwo$, $Y$ is a constrained predicate variable of the arity of $X$. There is one $\ern\foralltwo$ rule for each \emph{shape of atomic predicate} $\lambda\vec x.\hat A$. Such a shape fixes the head of $A$, which is a predicate variable or a higher-type constant, and which argument places are filled by which $x_i$. Each remaining argument place gets a variable of its own, and these variables are shared by all the occurrences of $X$.

\begin{example}[Second-order rules]\label{ex:NH-rules}
Take the shape $\foralltwo X(X(u_1) \to X(u_2))$ of Example~\ref{ex:layers}, with head $h$. Let $k$ be the head assigned to the shape $Y_1(u_1) \to Y_2(u_2)$, so that the counterpart of $R(a) \to R(b)$ is $k(R,a,R,b)$. The introduction rule for this shape, and its elimination rule for the shape of atomic predicate $\lambda x.Y(x)$, are
\[
\infer[\irn\foralltwo]{h(u_1,u_2)}{k(Y,u_1,Y,u_2)}\ \fresh{Y}
\qquad\qquad
\infer[\ern\foralltwo]{k(Y,u_1,Y,u_2)}{h(u_1,u_2)}
\]
The variable $Y$ fills both predicate places of $k$, because both come from occurrences of $X$.

For the shape of atomic predicate $\lambda x.Y(x,v)$, in which $Y$ is binary and $v$ is a further place, the elimination rule concludes $k'(Y,u_1,v,Y,u_2,v)$ from $h(u_1,u_2)$, where $k'$ is the head assigned to the shape $Y_1(u_1,u_2) \to Y_2(u_3,u_4)$. The variables $Y$ and $v$ are shared by the two occurrences of $X$. An instance of this rule passes from the counterpart of $\foralltwo X(X(a) \to X(a))$ to that of $R(a,c) \to R(a,c)$.
\end{example}

The instances of $\ern\foralltwo$ are therefore exactly the rules from $\flatop{(\foralltwo X\phi)}$ to $\flatop{(\phi[X \mapsto \lambda\vec x.A])}$, for closed $\foralltwo X\phi$ and $\lambda\vec x.A \in \Ain_n$. In short, the instances of $\base{N}_{\mathrm{H}}$ are the counterparts of the instances of the rules of $\NItwo$ for all closed formulae, except that $\ern\foralltwo$ is restricted to atomic predicates. The base $\base{N}_{\mathrm{H}}$ has infinitely many rules, but none of them mentions an individual or predicate constant, so it is a base.

The restriction of $\ern\foralltwo$ to atomic predicates loses nothing. The clause $(\foralltwo)$ quantifies over atomic predicates, and in $\base{N}_{\mathrm{H}}$ every condition has an atomic proxy, by Lemma~\ref{lem:flat-H}(d) and Corollary~\ref{cor:comprehensive}.

\begin{example}[The vixen again]\label{ex:NH-vixen}
The base $\base{N}_{\mathrm{H}}$ has its own word for `female fox'. Let $m$ be the head assigned to the shape $Y_1(u_1) \land Y_2(u_2)$. The counterpart of $F_e(a) \land F_o(a)$ is $m(F_e, a, F_o, a)$, and the rules of $\base{N}_{\mathrm{H}}$ for this shape are
\[
\infer{m(Y_1,u_1,Y_2,u_2)}{Y_1(u_1) & Y_2(u_2)}
\qquad
\infer{Y_1(u_1)}{m(Y_1,u_1,Y_2,u_2)}
\qquad
\infer{Y_2(u_2)}{m(Y_1,u_1,Y_2,u_2)}
\]
Their instances with $F_e$ for $Y_1$, $F_o$ for $Y_2$, and one name $a$ for both $u_1$ and $u_2$ are the instances of the three general rules of the vixen base $\base{V}$, with $m(F_e,a,F_o,a)$ in place of $V(a)$. So the atomic predicate $\lambda x.m(F_e,x,F_o,x)$ stands to `female fox' in $\base{N}_{\mathrm{H}}$ as `vixen' does in $\base{V}$. The simulation base contains rules of this kind for every condition of the language.
\end{example}

As in Section~\ref{sec:completeness}, we state the lemmas, derive the theorem, and prove the lemmas afterwards. There are three, one for each of the requirements listed at the start of this section.

The simulation lemma is the counterpart of Lemma~\ref{lem:dagger}. It now speaks of every closed formula of the language, and not only of those relevant to $\Gamma$ and $\zeta$, as comprehensiveness requires.

\begin{lemma}[Simulation]\label{lem:dagger-H}
For every base $\base{B}$ extending $\base{N}_{\mathrm{H}}$ and every closed formula $\gamma$: $\proves_{\base{B}} \flatop{\gamma}$ iff $\supp_{\base{B}} \gamma$.
\end{lemma}

Comprehensiveness follows at once. By the simulation lemma, the counterpart of each instance of a condition is supported exactly when the instance is. By Lemma~\ref{lem:flat-H}(d), those counterparts are the instances of a single atomic predicate.

\begin{corollary}[Comprehensiveness]\label{cor:comprehensive}
The base $\base{N}_{\mathrm{H}}$ is comprehensive, and so is every base extending it. In particular, every base $\base{B}$ has a comprehensive extension, namely $\base{B} \cup \base{N}_{\mathrm{H}}$.
\end{corollary}
\begin{proof}
Given $\lambda\vec x.\psi$, take $\lambda\vec x.A$ as in Lemma~\ref{lem:flat-H}(d). For every base $\base{C}$ extending $\base{N}_{\mathrm{H}}$ and all constants $\vec a$,
\[
\supp_{\base{C}} A[\vec x \mapsto \vec a] \quad\text{iff}\quad \proves_{\base{C}} \flatop{(\psi[\vec x \mapsto \vec a])} \quad\text{iff}\quad \supp_{\base{C}} \psi[\vec x \mapsto \vec a],
\]
by (At) and Lemma~\ref{lem:dagger-H}. So $\lambda\vec x.A$ is a proxy for $\lambda\vec x.\psi$ in $\base{N}_{\mathrm{H}}$. Since $\base{N}_{\mathrm{H}}$ mentions no individual or predicate constants, $\base{B} \cup \base{N}_{\mathrm{H}}$ is a base.
\end{proof}

The translation lemma is the counterpart of Lemma~\ref{lem:ddagger}, with the unwinding in place of the inverse of the counterpart map, and it is where comprehension enters the calculus. An application of $\ern\foralltwo$ in $\base{N}_{\mathrm{H}}$ instantiates by an atomic predicate. When that atomic predicate is a counterpart with gaps, its unwinding is a compound condition, so the translated step is an instance of full comprehension.

\begin{lemma}[Translation]\label{lem:ddagger-H}
If $\mathcal{D}$ is a derivation legal in $\base{N}_{\mathrm{H}}$, with conclusion $p$ and open assumptions $\atset{T}$, then $\natop{\atset{T}} \Ntwodash \natop{p}$.
\end{lemma}

\begin{proof}[Proof of Theorem~\ref{thm:B}]
Suppose $\Gamma \supp_{\mathrm{H}} \zeta$. By Corollary~\ref{cor:comprehensive}, $\base{N}_{\mathrm{H}}$ is comprehensive, so $\Gamma \supp_{\base{N}_{\mathrm{H}}} \zeta$.

The argument in the proof of Theorem~\ref{thm:A}, with Lemma~\ref{lem:dagger-H} in place of Lemma~\ref{lem:dagger}, gives $\flatop{\Gamma} \proves_{\base{N}_{\mathrm{H}}} \flatop{\zeta}$. That argument used only the simulation lemma and Lemma~\ref{lem:conservative}, and both are available here.

By Lemma~\ref{lem:ddagger-H}, $\natop{(\flatop{\Gamma})} \Ntwodash \natop{(\flatop{\zeta})}$, possibly from a subset of the assumptions. By Lemma~\ref{lem:flat-H}(b) and (c), this is $\Gamma \Ntwodash \zeta$.
\end{proof}

It remains to prove the simulation and translation lemmas. The proof of the simulation lemma follows that of Lemma~\ref{lem:dagger}. The one new case is the second-order quantifier, and its difficulty is the same as for the first-order quantifier: an eigenconstant must be found that satisfies~(G).

\begin{proof}[Proof of Lemma~\ref{lem:dagger-H}]
By induction on the degree of $\gamma$, for all bases extending $\base{N}_{\mathrm{H}}$ at once.

Clauses (i)--(vi) of Lemma~\ref{lem:atomic-clauses} hold for every base $\base{B}$ extending $\base{N}_{\mathrm{H}}$, with $p$ ranging over all basic sentences. Their proofs carry over, because each used only three things: that the base contains the rules of Figure~\ref{fig:N} for the shape in question, Lemma~\ref{lem:mono}, and the fact that a counterpart has the constants of its formula. The first holds by the definition of $\base{N}_{\mathrm{H}}$, and the third is Lemma~\ref{lem:flat-H}(a). None of the proofs used the injectivity of counterparts.

The new clause is
\begin{enumerate}[label=(\roman*)]\setcounter{enumi}{6}
  \item $\proves_{\base{B}} \flatop{(\foralltwo X\phi)}$ iff $\proves_{\base{B}} \flatop{(\phi[X \mapsto \lambda\vec x.A])}$ for every $\lambda\vec x.A \in \Ain_n$.
\end{enumerate}
Left to right is $\ern\foralltwo$. For the converse, choose a predicate constant $E$ fresh for $\base{B}$ and not in $\con{\foralltwo X\phi}$, and use the instance $\lambda\vec x.E(\vec x)$. There is a derivation of $\flatop{(\phi[X \mapsto E])}$ legal in $\base{B}$, without open assumptions, and $\irn\foralltwo$ with eigenconstant $E$ concludes $\flatop{(\foralltwo X\phi)}$. Condition~(G) holds as in clause (v): the premiss rests on no assumptions, $E$ does not occur in the rule instance, and $E$ is fresh for $\base{B}$.

The induction now proceeds as in the proof of Lemma~\ref{lem:dagger}. The cases for the constants of the first-order language are unchanged, since they used only clauses (i)--(vi), Lemma~\ref{lem:conservative}, and the induction hypothesis. In the case $\gamma = \foralltwo X\phi$, the instances by atomic predicates have smaller degree than $\gamma$. So clause (vii), the induction hypothesis, and $(\foralltwo)$ give the claim.
\end{proof}

The proof of the translation lemma replaces each basic sentence by its unwinding. Three things differ from Section~\ref{sec:completeness}: a basic sentence may unwind to a compound formula, discharges are no longer exact, and there are two new rules.

\begin{proof}[Proof of Lemma~\ref{lem:ddagger-H}]
By induction on $\mathcal{D}$, replacing each basic sentence $q$ by its unwinding $\natop{q}$ and using Lemma~\ref{lem:flat-H}(b).
\begin{itemize}[label=--]
  \item \emph{The rules of Figure~\ref{fig:N}.} Each translates into the rule of $\NItwo$ with the same name, as in the proof of Lemma~\ref{lem:ddagger}. By Lemma~\ref{lem:flat-H}(b), the unwinding of the counterpart of a formula is the unwinding of the formula, and unwinding commutes with the connectives, the quantifiers, and instantiation by constants. When the instance of $Z$ is a basic sentence $q$ with an assigned head, its unwinding $\chi := \natop{q}$ is a compound formula, which $\ern\lor$, $\ern\exists$, and $\efq$ allow.
  \item \emph{Rules that discharge hypotheses} ($\irn\to$, $\ern\lor$, and $\ern\exists$). Injectivity is no longer available, but nothing is lost. If a premiss derivation rests on assumptions $\atset{T}'$ and the rule discharges $q$, then the translation rests on $\natop{\atset{T}'}$, which is contained in $\natop{(\atset{T}' \setminus \{q\})} \cup \{\natop{q}\}$. So the corresponding rule of $\NItwo$, discharging $\natop{q}$, gives the required derivation, possibly with fewer open assumptions.
  \item \emph{Eigenvariable conditions of the first-order rules.} These transfer because the unwinding preserves constants. In the case of $\ern\exists$ with eigenconstant $e$, every undischarged assumption $q$ of the minor premiss has $e \notin \con{q} = \con{\natop{q}}$. So the condition of $\NItwo$ holds even if $\natop{q}$ happens to coincide with the discharged $(\natop{\phi})[x \mapsto e]$. And $e$ occurs neither in the unwinding of the conclusion nor in that of $\flatop{(\exists x\phi)}$.
  \item $\ern\foralltwo$ \emph{with atomic predicate} $\lambda\vec x.A$. The application translates to the inference from $\natop{(\foralltwo X\phi)}$ to $(\natop{\phi})[X \mapsto \lambda\vec x.\natop{A}]$. When $A$ has an assigned head, $\natop{A}$ is a compound formula and this is an instance of comprehension. Otherwise it is an instance by an atomic predicate. Either way it is an application of $\ern\foralltwo$ in $\NItwo$.
  \item $\irn\foralltwo$ \emph{with eigenconstant} $E$. The application translates to $\irn\foralltwo$ in $\NItwo$. The eigenvariable condition holds by~(G) and Lemma~\ref{lem:flat-H}(b). If $X$ is vacuous, there is no eigenconstant, and we take for $E$ a predicate constant occurring nowhere in the translated derivation. \qedhere
\end{itemize}
\end{proof}

Theorems~\ref{thm:A} and~\ref{thm:B} have the same form and the same method of proof. They differ in what the quantifiers range over, in the class of bases over which validity is taken, and in the machinery that the semantics and the proof require, and Table~\ref{tab:comparison} sets the differences out.

\begin{table}[t]
\caption{The two theorems}\label{tab:comparison}
\medskip
\small
\begin{tabular}{@{}>{\raggedright\arraybackslash}p{0.19\linewidth}>{\raggedright\arraybackslash}p{0.36\linewidth}>{\raggedright\arraybackslash}p{0.37\linewidth}@{}}
\hline
 & Theorem~\ref{thm:A} & Theorem~\ref{thm:B} \\
\hline
Range of quantification & $\forall x$, $\exists x$: the individual constants & $\forall x$, $\exists x$: the individual constants; $\foralltwo X$: the atomic predicates $\Ain_n$ \\[1mm]
Class of bases & all bases; alternatively, all finite bases & the comprehensive bases \\[1mm]
Semantic machinery & rules with constraints; freshness condition~(G) & in addition: atomic predicates, proxies, higher-type atoms \\[1mm]
Simulation base & $\base{N}$: finite; rules for the shapes of $\Gamma$ and $\zeta$ & $\base{N}_{\mathrm{H}}$: infinite; rules for all shapes; comprehensive \\[1mm]
Counterparts & injective; inverted by $\natop{(-)}$ & not injective; unwound by $\natop{(-)}$; serve as proxies \\[1mm]
Result & $\Gamma \supp \zeta$ iff $\Gamma \Ndash \zeta$ & $\Gamma \supp_{\mathrm{H}} \zeta$ iff $\Gamma \Ntwodash \zeta$ \\[1mm]
Logic characterized & intuitionistic first-order logic & intuitionistic Henkin second-order logic with full comprehension; not full second-order semantics \\
\hline
\end{tabular}
\end{table}

The restriction to comprehensive bases in Theorem~\ref{thm:B} enters through soundness, not through completeness. Completeness needs only one suitable base, and $\base{N}_{\mathrm{H}}$ is comprehensive. Soundness needs every base in the class to validate $\ern\foralltwo$, and a base without proxies does not (Example~\ref{ex:not-sound}).

Although the two theorems speak of different classes of bases, they agree where they overlap.

\begin{corollary}[First-order formulae]\label{cor:fo-agree}
For a finite set $\Gamma$ of closed first-order formulae and a closed first-order formula $\zeta$, the following are equivalent: $\Gamma \supp_{\mathrm{H}} \zeta$; $\Gamma \Ndash \zeta$; and $\Gamma \supp_{\base{B}} \zeta$ for every base $\base{B}$ of the second-order language.
\end{corollary}
\begin{proof}
If $\Gamma \Ndash \zeta$, then $\Gamma \supp_{\base{B}} \zeta$ for every base of the second-order language, by Theorem~\ref{thm:soundness}, whose proof is unaffected by the larger language. This implies $\Gamma \supp_{\mathrm{H}} \zeta$. If $\Gamma \supp_{\mathrm{H}} \zeta$, then $\Gamma \Ntwodash \zeta$ by Theorem~\ref{thm:B}. It remains to show that $\NItwo$ is conservative over $\NI$ for first-order formulae.

This is the standard Kripke-model argument, which we sketch. Semantics of this kind for intuitionistic second-order logic go back to Prawitz~\cite{Prawitz1970second}. Take a Kripke countermodel to $\Gamma \vdash_{\NI} \zeta$~\cite{kripke1965semantical}. At each node $k$, let $\foralltwo X$ range over the \emph{persistent relations} above $k$, which assign to each node $j \geq k$ a relation on the domain of $j$, increasing along the order; and let no higher-type atom be forced. Forcing of first-order formulae is unchanged.

The rule $\irn\foralltwo$ is sound for the same reason as $\irn\forall$: forcing at a node depends only on the nodes above it, so the eigenconstant may be reinterpreted there by any persistent relation. The rule $\ern\foralltwo$ is sound because, by persistence of forcing, the relation defined by any formula $\psi$ is persistent. Hence $\Gamma \nvdash_{\NItwo} \zeta$.
\end{proof}

\begin{remark}[Higher-type constants]\label{rem:H}
The higher-type constants are part of the object language. The semantics needs them, since counterparts must be available as proxies, but they do not change the logic. Derivability in $\NItwo$ is preserved by substituting formulae for higher-type atoms, for instance by replacing every higher-type atom in a derivation by $\bot$. So for formulae without higher-type constants, Theorem~\ref{thm:B} characterizes derivability in the language without them.
\end{remark}

\section{Discussion}\label{sec:discussion}

We discuss, in turn, Sandqvist's worry about substitutional quantification, the interpretation of comprehensiveness, the relation to full second-order semantics, and Dummett's demand that definitions be conservative.

Sandqvist's objection to substitutional quantification was recalled in the introduction. In \S3.3.1 of~\cite{Sandqvist2005inferentialist} he replies on behalf of a refined theory, whose bases are \emph{finite} and may contain open rules. If a finite base supports $\forall x([bx] \supset [cx])$, the inference from $[ba_0]$ to $[ca_0]$ must hold for some name $a_0$ the base does not mention. This can only be in virtue of open rules, and speakers plausibly accept such rules only if they hold of all objects, named or not. He calls this `a vague sort of argument', and doubts that it generalizes to more complex universals.

The semantics of this paper is carried out in a setting close to Sandqvist's. Bases mention only finitely many names, Theorem~\ref{thm:A} holds for finite bases, and generality is carried by open rules. The setting makes his argument precise, in four steps.
\begin{itemize}[label=--]
  \item \emph{The atomic case is the Renaming Lemma.} If $e$ does not occur in $\base{B}$ and $\proves_{\base{B}} P(e)$, then $\proves_{\base{B}} P(a)$ for every constant $a$, named in the base or not (Lemma~\ref{lem:renaming}). In a derivation about a name the base does not mention, every rule applied to that name is an instance of a general rule, so the derivation is one for an arbitrary object. This is Sandqvist's observation that the relevant chain of rule-applications `can be hypothetically replicated for an arbitrary object, named or not'~\cite[\S3.3.1]{Sandqvist2005inferentialist}, now as a lemma.
  \item \emph{The quantifier clause is substitutional, but never an enumeration.} By Corollary~\ref{cor:fresh-instance}, a base supports $\forall x\phi$ iff it supports the instances at the finitely many names occurring in the base or in $\phi$, and the instance at one name it does not mention. Beyond those finitely many named cases, support of a universal rests on the base's general content: its rules with variables, which apply to every name alike.
  \item \emph{Complex universals.} Corollary~\ref{cor:fresh-instance} holds for every formula, however complex, and in every base. Moreover, by Lemma~\ref{lem:dagger}, in every extension of the simulation base, support of a universal formula coincides with derivability of its counterpart, which is obtained from a single premiss established for one fresh name. This answers Sandqvist's doubt about generalizing to complex universals. Lemma~\ref{lem:dagger-H} extends the answer to second-order universals, and to every closed formula of the language, in the extensions of $\base{N}_{\mathrm{H}}$.
  \item \emph{Freshness for the base is the hinge.} Part~\ref{G3} of condition~(G) ensures that only names about which the base says nothing in particular are treated as arbitrary (Example~\ref{ex:fresh}).
\end{itemize}

The second step is special to the first-order quantifier. Over arbitrary bases, Corollary~\ref{cor:fresh-instance} has no analogue for $\foralltwo$: the base of Example~\ref{ex:names-alone} supports the instance of $\foralltwo X^0 X^0$ at every $0$-ary predicate constant, mentioned or not, and does not support $\foralltwo X^0 X^0$. A predicate name that the base does not mention stands for an arbitrary predicate \emph{name}, not for an arbitrary atomic predicate.

The quantifiers range over the names of the language, and the semantics says nothing about what those names refer to. What the present account adds is that a name the base does not mention behaves as an arbitrary object.

We turn to comprehensiveness. Theorem~\ref{thm:B}, Corollary~\ref{cor:comprehensive}, and Example~\ref{ex:not-sound} are mathematical results: $\NItwo$ is sound and complete for support in comprehensive bases, every base is contained in a comprehensive one, and comprehension is not sound outside them. What comprehensiveness \emph{means} is a matter of interpretation, and we try to keep the two apart.

Three readings suggest themselves. On the first, comprehensiveness is \emph{conceptual availability}: for each condition, the base already has a predicate governed by rules that define it. On the second, it is \emph{definitional capacity}: such a predicate could be introduced by extending the base if the occasion arose. On the third, it is \emph{conceptual competence}: any condition can be treated as a unit, and put wherever a predicate may go.

Definition~\ref{def:proxy} states the first reading. A proxy is one of the atomic predicates of the base, and what makes it a proxy are rules that the base contains. The third reading then follows from the first: by Lemma~\ref{lem:replacement}, in a base that has a proxy for a condition, the two may be exchanged anywhere. The second reading describes a different and weaker property, which every base has, by Corollary~\ref{cor:comprehensive}.

The difference between the first and second readings is the difference between possessing a concept and being able to acquire it, and the semantics is sensitive to it. Support is evaluated in a base and in its extensions. An extension may add definitions, but nothing requires it to, so a base that lacks a proxy for a condition is not treated as if it had one. The base of Example~\ref{ex:not-sound} is one definition short of supporting $(R \to R) \to Q$, and it does not support it. H-validity is therefore validity in bases that have the concepts, not in bases that could acquire them. Henkin's condition is of the same kind: a Henkin structure contains its definable relations, and is not merely extendible to a structure that does.

On this account the impredicativity of comprehension lies in the base. The clause $(\foralltwo)$ is a well-founded recursion, which explains $\foralltwo X\phi$ by instances of lower degree, and this much is mathematics. Comprehension nevertheless holds for conditions $\psi$ that themselves quantify over all atomic predicates, the proxy for $\psi$ among them. The circle is still there, in the requirement that the atomic predicates of a comprehensive base be closed under conditions that quantify over those very predicates. The impredicativity is thus a property of the class of bases over which validity is taken, as for Henkin it is a property of the class of structures.

We do not claim that this makes comprehension acceptable to a predicativist. The existence of a comprehensive base is proved by defining $\base{N}_{\mathrm{H}}$ over all shapes of formulae at once, and a predicativist may object to that definition. Our claim is only about where the circularity sits: the meaning of the quantifier is given without it, and it reappears as a condition on bases.

The requirement that a proxy agree with its condition in \emph{every} extension has a mathematical role and a conceptual one. Mathematically, it is what the Replacement Lemma needs (Example~\ref{ex:coincidence}), and it makes comprehensiveness persist under extension. Conceptually, it separates definition from coincidence. In the base $\base{V}'$, `vixen' and `female fox' apply to the same individuals as far as the base records, whereas in $\base{V}$ no further information could separate them, which is what one expects of a definition.

The semantics of this paper recovers the intuitionistic version of Henkin's logic, not full second-order logic. V\"a\"an\"anen~\cite{Vaananen2001} has argued that, once second-order logic is formalized, the distinction between Henkin and full semantics effectively vanishes. From this point of view, Henkin's logic is what one should expect of a semantics in terms of proofs. Whether some class of bases, or some other choice of clause, recovers full second-order consequence is open.

A last question concerns definitions. Dummett requires that introducing new vocabulary by harmonious rules be \emph{conservative}: it should not allow one to derive anything in the old vocabulary that could not be derived before~\cite[pp.~217--219]{Dummett1991logical}. (Lemma~\ref{lem:conservative} is a different matter, since it concerns rules and assumptions.) The semantics of Section~\ref{sec:second-order} does not need conservativeness, since validity is support in comprehensive bases. Still, it is natural to ask whether adding definitions to an arbitrary base is conservative over it.

In general it is not. Rules with variables may give a base general content about all predicates, or about all sentences, and that content then applies to a newly coined predicate. The base $\base{B}$ of Example~\ref{ex:not-sound}, which derives nothing, is a case. Add a new $0$-ary predicate $T$, defined to mean $R \to R$ by the rules
\[
\infer{T}{\deduce{R}{[R]}}
\qquad\qquad
\infer{R}{T & R}
\]
The extended base derives $T$, by discharging the assumption $R$, and hence derives $Q$, by the rule of $\base{B}$ with $Z := T$. But $Q$ belongs to the old vocabulary and was not derivable before.

The rule of $\base{B}$ by which $Q$ follows from whatever is assertible was idle while nothing was assertible outright, and the definition activates it. The question of conservativeness must therefore be asked of restricted classes of bases, for instance those whose rules contain no atom variables or predicate variables. We have not settled it for any such class.

Several questions remain, the first about finiteness. The simulation base $\base{N}_{\mathrm{H}}$ is infinite, and we conjecture that no finite base is comprehensive: a finite set of rules covers only finitely many predicate arities and higher-type signatures. A finite comprehensive base would plausibly need \emph{compositional} counterparts, such as $\flatop{(\phi\to\psi)} := \mathsf{imp}(\flatop{\phi},\flatop{\psi})$, with atoms taking sentences as arguments. It would be an interpreter for second-order logic written in the logic itself, in the style of $\lambda$Prolog~\cite{MillerNadathur2012}.

Base-extension semantics gives an intuitively constructive account of logical consequence, but its computational content, and its relationship with proof-theoretic validity for a given logic, remain to be understood. The close relationship between intuitionistic second-order logic and the polymorphic $\lambda$-calculus (see, for example, \cite{TroelstraSchwichtenberg2000}) suggests that the semantics of this paper may be a useful starting point.

For arithmetic, proof-theoretic semantics already gives a simple argument for the consistency of first-order Peano arithmetic~\cite{gheorghiu2025pa}. The present semantics may support a similar analysis of second-order arithmetic.

\bibliographystyle{abbrvnat}
\bibliography{bib}

\end{document}